\documentclass[12pt,a4paper,fullpage]{amsart}
\usepackage{amsmath,amsthm}
\usepackage{amssymb}
\usepackage{amsbsy,amsfonts,latexsym,
                        amsopn,amstext,amsxtra,euscript,amscd,mathrsfs,color,bm,cite}
\usepackage{todonotes}
\newtheorem{theorem}{Theorem}[section]
\newtheorem*{theorem*}{Theorem}
\newtheorem{proposition}[theorem]{Proposition}
\newtheorem{lemma}[theorem]{Lemma}
\newtheorem{Remark}[theorem]{Remarks}
\newtheorem{corollary}[theorem]{Corollary}
\newtheorem{Conjecture}[theorem]{Conjecture}
\makeatletter
\DeclareRobustCommand*\cal{\@fontswitch\relax\mathcal}

\title[Mean square values of non primitive $L$- functions]{{\bf Mean square values of $L$-functions over subgroups\\ 
for non primitive characters, Dedekind sums \\
and bounds on relative class numbers}}
\author{
St\'ephane R. LOUBOUTIN}
\address{Aix Marseille Universit\'e, CNRS, Centrale Marseille, I2M,\
Marseille, FRANCE}
\email{stephane.louboutin@univ-amu.fr}
\author{Marc MUNSCH}
\address{Unige, Dipartimento di Matematica, 
Genova, Italia }
\email{munsch@dima.unige.it}
\date{\today}

\begin{document}
\bibliographystyle{alpha}
\begin{abstract} An explicit formula for the mean value of $\vert L(1,\chi)\vert^2$ is known, 
where $\chi$ runs over all odd primitive Dirichlet characters of prime conductors $p$. 
Bounds on the relative class number of the cyclotomic field ${\mathbb Q}(\zeta_p)$ follow.
Lately the authors obtained that the mean value of $\vert L(1,\chi)\vert^2$ is asymptotic to $\pi^2/6$, 
where $\chi$ runs over all odd primitive Dirichlet characters of prime conductors $p\equiv 1\pmod{2d}$
which are trivial on a subgroup $H$ of odd order $d$ of the multiplicative group $({\mathbb Z}/p{\mathbb Z})^*$, 
provided that $d\ll\frac{\log p}{\log\log p}$.
Bounds on the relative class number of the subfield of degree $\frac{p-1}{2d}$ 
of the cyclotomic field ${\mathbb Q}(\zeta_p)$ 
follow. 
Here, for a given integer $d_0>1$ 
we consider the same questions for the non-primitive odd Dirichlet characters $\chi'$ modulo $d_0p$ 
induced by the odd primitive characters $\chi$ modulo $p$. 
We obtain new estimates for Dedekind sums 
and deduce that the mean value of $\vert L(1,\chi')\vert^2$ is asymptotic to 
$\frac{\pi^2}{6}\prod_{q\mid d_0}\left (1-\frac{1}{q^2}\right )$, 
where $\chi$ runs over all odd primitive Dirichlet characters of prime conductors $p$
which are trivial on a subgroup $H$ of odd order $d\ll\frac{\log p}{\log\log p}$.
As a consequence we improve the previous bounds 
on the relative class number of the subfield of degree $\frac{p-1}{2d}$ of the cyclotomic field ${\mathbb Q}(\zeta_p)$. 
Moreover, we give a method to obtain explicit formulas 
and use Mersenne primes to show that our restriction on $d$ is essentially sharp. \end{abstract}

\maketitle

\footnotetext{

Mean square value. 
Relative class number. 
Dedekind sums. 
Cyclotomic field.}

\section{Introduction}
Let $X_f$ be the multiplicative group of the $\phi (f)$ Dirichlet characters modulo $f>2$. 
Let 
$X_f^-
=\{\chi\in X_f;\ \chi (-1)=-1\}$ be the set 
of the $\phi (f)/2$ odd Dirichlet characters modulo $f$. 
Let $L(s,\chi)$ 
be the Dirichlet $L$-function associated with $\chi\in X_f$. 
Let $H$ denote a subgroup of index $m$ in the multiplicative group $G:=({\mathbb Z}/f{\mathbb Z})^*$. 
We assume that $-1\not\in H$. 
Hence $m$ is even.
We set 
$X_f(H)
=\{\chi\in X_f;\ \chi_{/H}=1\}$, 
a subgroup of order $m$ of $X_f$ isomorphic to the group of Dirichlet characters of the abelian quotient group $G/H$ 
of order $m$. 
Define $X_f^-(H)=\{\chi\in X_f^-;\ \chi_{/H}=1\}$, 
a set of cardinal $m/2$.
Let $K$ be an abelian number field of degree $m$ and prime conductor $p\geq 3$, 
i.e. let $K$ be a subfield of the cyclotomic number field ${\mathbb Q}(\zeta_p)$ 
(Kronecker-Weber's theorem). 
The Galois group ${\rm Gal}({\mathbb Q}(\zeta_p)/{\mathbb Q})$ 
is canonically isomorphic to the multiplicative cyclic group $({\mathbb Z}/p{\mathbb Z})^*$ 
and
$H
:={\rm Gal}({\mathbb Q}(\zeta_p)/K)$
is a subgroup of $({\mathbb Z}/p{\mathbb Z})^*$ of index $m$ and order $$d=(p-1)/m.$$ 
Now, assume that $K$ is imaginary. 
Then $d$ is odd, 
$m$ is even,
$-1\not\in H$ 
and the set 
$$X_K^-
:=X_p^-(H)
:=\{\chi\in X_p^-;\hbox{ and } \chi_{/H}=1\}$$
is of cardinal $(p-1)/(2d) =m/2$.
Let $K^+$ be the maximal real subfield of $K$ of degree $m/2$ fixed by the complex conjugation. 
The class number $h_{K^+}$ of $K^+$ divides the class number $h_K$ of $K$. 
The relative class number of $K$ is defined by $h_K^- =h_K/h_{K^+}$. 
We refer the reader to \cite{Ser} and \cite{Was} for such basic knowledge. 
The mean square value of $L(1,\chi)$ as $\chi$ ranges in $X_f^-(H)$ is defined by
\begin{equation}\label{M(f,H)}
M(f,H)
:={1\over\# X_f^-(H)}\sum_{\chi\in X_f^-(H)}\vert L(1,\chi)\vert^2.
\end{equation}
The analytic class number formula and the arithmetic-geometric mean inequality give 
\begin{equation}\label{boundhKminus}
h_K^-
=w_K\left (\frac{p}{4\pi^2}\right )^{m/4}\prod_{\chi\in X_K^-} L(1,\chi)
\leq w_K\left ({pM(p,H)\over 4\pi^2}\right )^{m/4},
\end{equation}
where $w_K$ is the number of complex roots of unity in $K$. 
Hence $w_K=2p$ for $K ={\mathbb Q}(\zeta_p)$ and $w_K=2$ otherwise. 
In \cite[Theorem 1.1]{LMQJM} we proved that 
\begin{equation}\label{asymptoticMpH}
M(p,H) 
=\frac{\pi^2}{6}+o(1)
\end{equation}
as $p$ tends to infinity uniformly over subgroups $H$ of $({\mathbb Z}/p{\mathbb Z})^*$ 
of odd order $d\leq\frac{\log p}{3(\log\log p)}$ \footnote{This restriction on $d$ is probably optimal, by (\ref{Mpd}).}. 
Hence, by \eqref{boundhKminus} we have 
\begin{equation}\label{boundhKgen} 
h_{K}^{-} 
\leq w_K\left(\frac{(1+o(1))p}{24}\right)^{(p-1)/4d}.
\end{equation} 
In some situations it is even possible to give an explicit formula for $M(p,H)$ 
implying a completely explicit bound for $h_K^-$. 
Indeed, by \cite{W} and \cite{Met} (see also (\ref{M(f,1)})), 
we have
\begin{equation}\label{Mp1}
M(p,\{1\})
={\pi^2\over 6}\left (1-{1\over p}\right )\left (1-{2\over p}\right )
\leq\frac{\pi^2}{6}
\ \ \ \ \ \hbox{($p\geq 3$)}.
\end{equation}
Hence, 
\begin{equation}\label{boundhpminus24}
h_{{\mathbb Q}(\zeta_p)}^-
\leq 2p\left ({pM(p,\{1\})\over 4\pi^2}\right )^{(p-1)/4}
\leq 2p\left ({p\over 24}\right )^{(p-1)/4}.
\end{equation} 
We refer the reader to \cite{Gra} for more information about the expected size of $h_{{\mathbb Q}(\zeta_p)}^-$. 
The only other situation where a similar explicit result is known is the following one
(see Theorem \ref{thp3M2} for a new proof).

\begin{theorem*}
(See \footnote{Note the misprint in the exponent in \cite[(8)]{LouBPASM64}.} \cite[Theorem 1]{LouBPASM64}).
Let $p\equiv 1\pmod 6$ be a prime integer. 
Let $K$ be the imaginary subfield of degree $(p-1)/3$ of the cyclotomic number field ${\mathbb Q}(\zeta_p)$.
Let $H$ be the subgroup of order $3$ of the multiplicative group $({\mathbb Z}/p{\mathbb Z})^*$. 
We have (compare with (\ref{Mp1}) and (\ref{boundhpminus24}))
\begin{equation}\label{boundthp3}
M(p,H)
={\pi^2\over 6}\left (1-{1\over p}\right )
\leq\frac{\pi^2}{6}
\hbox{ and }
h_K^-\leq 2\left (\frac{p}{24}\right )^{(p-1)/12}.\end{equation}
\end{theorem*} 

In \cite{LouCMB36/37} 
(see also \cite{LouPMDebr78}), the following simple argument allowed to improve on \eqref{boundhpminus24}.
Let $d_0>1$ be a given integer. 
Assume that $\gcd (d_0,f)=1$. 
For $\chi$ modulo $f$ let $\chi'$ be the character modulo $d_0f$ induced by $\chi$. 
Then, 
\begin{equation}\label{L1XXprime}
L(1,\chi) 
=L(1,\chi')\prod_{q\mid d_0}\left (1-\frac{\chi(q)}{q}\right )^{-1}
\end{equation}
(throughout the paper this notation means that $q$ runs over the distinct prime divisors of $d_0$). 
Let $H$ be a subgroup of order $d$ of the multiplicative group $({\mathbb Z}/f{\mathbb Z})^*$, 
with $-1\not\in H$. 
We define
\begin{equation}\label{Md0(f,H)}
M_{d_0}(f,H)
:={1\over\# X_f^-(H)}\sum_{\chi\in X_f^-(H)}\vert L(1,\chi')\vert^2
\end{equation}
and\footnote{Note that $\Pi_{d_0}(f,H)\in {\mathbb Q}_+^*$, by Lemma \ref{formulaPi}.}
\begin{equation}\label{Pi}
\Pi_{d_0}(f,H)
:=\prod_{q\mid d_0}\prod_{\chi\in X_f^-(H)}\left (1-\frac{\chi(q)}{q}\right )
\text{ and }
D_{d_0}(f,H)
:=\Pi_{d_0}(f,H)^{4/m}.
\end{equation}
Clearly there is no restriction in assuming from now on that $d_0$ is square-free. 
Let now $H$ be of odd order $d$ in the multiplicative group $({\mathbb Z}/p{\mathbb Z})^*$.
Using (\ref{L1XXprime}), we obtain (compare with (\ref{boundhKminus})):
\begin{equation}\label{boundhKminusd0}
h_K^-
=\frac{w_K}{\Pi_{d_0}(p,H)}
\left (\frac{p}{4\pi^2}\right )^{m/4}\prod_{\chi\in X_K^-}L(1,\chi')
\leq w_K\left (\frac{pM_{d_0}(p,H)}{4\pi^2D_{d_0}(p,H)}\right )^{m/4}.
\end{equation} 
Let $d=o(\log p)$ as $p\rightarrow\infty$. 
Then, by Corollary \ref{PiKminus} below, 
we have
$$D_{d_0}(p,H) =1+o(1)$$
and we expect that 
\begin{equation}\label{asymptoticMd0(p,H)M(p,H)}
M_{d_0}(p,H)
\sim\left\{\prod_{q\mid d_0}\left (1-\frac{1}{q^2}\right )\right\}
\times M(p,H).
\end{equation} 
Hence, (\ref{boundhKminusd0}) should indeed improve on (\ref{boundhKminus}).
The aim of this paper is two-fold. Firstly, in Theorem \ref{asympd0} we give an asymptotic formula for $M_{d_0}(p,H)$ 
when $d$ satisfies the same restriction as in \eqref{asymptoticMpH} 
allowing us to improve on the bound \eqref{boundhKgen}. 
Secondly we treat the case of groups of order $1$ and $3$ for small $d_0$'s 
as well as the case of Mersenne primes and groups of size $\approx \log p$. 
In both cases an explicit description of these subgroups allows us to obtain explicit formulas for $M_{d_0}(p,H)$. 
Our main result is the following.

\noindent\frame{\vbox{
\begin{theorem}\label{asympd0}
Let $d_0\geq 1$ be a given square-free integer. 
As $p\rightarrow +\infty$ we have the following asymptotic formula
$$M_{d_0}(p,H)
=\frac{\pi^2}{6}\prod_{q\mid d_0}\left (1-{1\over q^2}\right )
+O(d(\log p)^2 p^{-\frac{1}{d-1}})
=\frac{\pi^2}{6}\prod_{q\mid d_0}\left (1-{1\over q^2}\right )
+o(1)$$ 
uniformly over subgroups $H$ of $({\mathbb Z}/p{\mathbb Z})^*$ of odd order $d\leq\frac{\log p}{3(\log\log p)}.$ 
Moreover, let $K$ be an imaginary abelian number field of prime conductor $p$ and of degree $m=(p-1)/d$. Let $C<4\pi^2=39.478..$ be any positive constant. If $p$ is sufficiently large and $m\geq 3\frac{(p-1) \log\log p} {\log p}$, then we have \begin{equation}\label{boundhKopt} 
h_K^-
\leq w_K\left (\frac{p}{C}\right )^{(p-1)/4d}.
\end{equation}
\end{theorem}
}}

\begin{Remark} The second result in Theorem \ref{asympd0} improves on \eqref{boundhKgen}, \eqref{boundhpminus24} and \eqref{boundthp3}. It follows from the first result in Theorem \ref{asympd0}, and by using \eqref{boundhKminusd0} and \eqref{behaviorPidexplicit}, where we take $d_0$ as the product of sufficiently many consecutive first primes. \\

The special case $d_0=1$ was proved in \cite[Theorem $1.1$]{LMQJM}. 
Note that the restriction on $d$ cannot be extended further to the range $d=O(\log p)$ as shown by Theorem \ref{TheoremMersenne}. Moreover the constant $C$ in \eqref{boundhKopt} cannot be taken larger than $4\pi^2$, 
see the discussion about Kummer's conjecture in \cite{MP01}.
\end{Remark} 

In the first part of the paper, the presentation goes as follows:
\begin{itemize}
\item In Section \ref{genlemmas}, 
we explain the condition about the prime divisors of $d_0$ and prove that $D_{d_0}(p,H) =1+o(1)$.
\item In Section \ref{Dedekind}, we review some results on Dedekind sums 
and prove a new bound of independent interest for Dedekind sums $s(h,f)$ with $h$ being of small order modulo $f$ 
(see Theorem \ref{indivbound}). 
To do so we use techniques from uniform distribution and discrepancy theory.
Then we relate $M_{d_0}(p,H)$ to twisted moments of $L$- functions which we further express in terms of Dedekind sums. 
For the sake of clarity, we first treat separately the case 
$H=\{1\}$. Note that we found that this case is related to elementary sums of maxima that we could not estimate directly, 
see Section \ref{sectionsummax}.
Using our estimates on Dedekind sums we deduce the asymptotic formula of Theorem \ref{asympd0} 
and the related class number bounds.
\end{itemize} In the second part of the paper, we focus on the explicit aspects. 
Let us describe briefly our presentation:
\begin{itemize}
\item In Section \ref{Casetrivialgroup} we establish a formula for $M_{d_0}(f,\{1\})$, $d_0>2$,
provided that all the prime factors $q$ of $f$ satisfy $q\equiv\pm 1\pmod{d_0}$. 
In particular, we get formulas for $M_{d_0}(f,\{1\})$ for $d_0\in\{1,2,3,6\}$ and $\gcd (d_0,f)=1$ 
(such formulae become harder to come by as $d_0$ gets larger). 
For example, for $p\geq 5$ and $d_0=6$, 
using Theorem \ref{M3M6}
we obtain the following formula for $M_6(p,\{1\})$:
$$M_6(p,\{1\})
={\pi^2\over 9}\left (1-\frac{c_p}{p}\right )
\leq\frac{\pi^2}{9},
\hbox{ where }
c_p
=\begin{cases}
1&\hbox{if $p\equiv 1\pmod 3$}\\
0&\hbox{if $p\equiv 2\pmod 3$}\\
\end{cases}$$ 
which by (\ref{boundhKminusd0}) and Corollary \ref{PiKminus} give improvements on (\ref{boundhpminus24}) 
(see also \cite{Feng} and \cite{LouCMB36/37}) 
$$h_{{\mathbb Q}(\zeta_p)}^-
\leq 3p\left ({p\over 36}\right )^{(p-1)/4}.$$
See also \cite[Theorem 5.2]{LouFetA??} for even better bounds.\\
In Section \ref{sectionformulad0} we obtain an explicit formula of the form 
\begin{equation}\label{Md0pHintro}
M_{d_0}(p,H)
=\frac{\pi^2}{6}
\left\{\prod_{q\mid d_0}\left (1-\frac{1}{q^2}\right )\right\}
\left (1+\frac{N_{d_0}(p,H)}{p}\right ),
\end{equation}
where $N_{d_0}(p,H)$ defined in (\ref{defNd0fH}) 
is an explicit average of Dedekind sums. 
In Proposition \ref{H=1} we prove that $N_{d_0}(p,\{1\})\in {\mathbb Q}$ depends only on $p$ modulo $d_0$ 
and is easily computable.
\item For $H\neq\{1\}$ explicit formulae for $M_{d_0}(p,H)$ seem difficult to come by. 
In Section \ref{genMersenne}, 
we focus on Mersenne primes $p=2^d-1$, 
with $d$ odd. We take $H=\{2^k;\ 0\leq k\leq d-1\}$, 
a subgroup of odd order $d$ of the multiplicative group $({\mathbb Z}/p{\mathbb Z})^*$. 
For $d_0\in\{1,3,15\}$ we prove in Theorem \ref{Mersenne} that
$$M_{d_0}(p,H)
=\frac{\pi^2}{2}\left\{\prod_{q\mid d_0}\left (1-\frac{1}{q^2}\right )\right\}
\left (1+\frac{N_{d_0}'(p,H)}{p}\right ),$$
where $N_{d_0}'(p,H)=a_1(p)d+a_0(p)$ 
with $a_1(p),a_0(p)\in {\mathbb Q}$ depending only on $p=2^d-1$ modulo $d_0$ 
and easily computable. 
In the range $d \gg \log p$, 
we see that $M_{d_0}(p,H)$ has a different asymptotic behavior than the one in Theorem \ref{asympd0}.
\item In Section \ref{Sectiond=3}, we turn to the specific case of subgroups of order $3$. 
Writing $f=a^2+ab+b^2$ not necessarily prime, 
and taking $H=\{1,a/b,b/a\}$, 
the subgroup of order $3$ of the multiplicative group $({\mathbb Z}/f{\mathbb Z})^*$, 
we prove in Proposition \ref{boundsSab} that $N_{d_0}(f,H) =O(\sqrt f)$ in (\ref{Md0pHintro}) for $d_0\in\{1,2,3,6\}$. 
To do so we obtain bounds for the Dedekind sums stronger than the one in Theorem \ref{indivbound}. 
Note that this cannot be expected in general for subgroups of order $3$ modulo composite $f$ 
(see Remark \ref{compositebound} and \ref{dedekindtwosizes}).
Furthermore we show that these bounds are sharp in the case of primes $p=a^2+a+1$, 
in accordance with Conjecture \ref{conjDedekind}.
\end{itemize}

\section{Preliminaries}\label{genlemmas}
\subsection{Algebraic considerations}
Take $a\in {\mathbb Z}$ with $\gcd (a,f)=1$.
There are infinitely many prime integers in the arithmetic progressions $a+f{\mathbb Z}$. 
Taking a prime $p\in a+f{\mathbb Z}$ with $p>d_0f$, 
we have $s_{d_0}(p)=a$, 
where 
$s_{d_0}:({\mathbb Z}/d_0f{\mathbb Z})^*\longrightarrow ({\mathbb Z}/f{\mathbb Z})^*$ 
is the canonical morphism. 
Therefore, $s_{d_0}$ surjective
and its kernel is of order $\phi(d_0)$. 
Let $H$ be a subgroup of $({\mathbb Z}/f{\mathbb Z})^*$ of order $d$.
Then $H_{d_0} =s_{d_0}^{-1}(H)$ is a subgroup of order $\phi(d_0)d$ of $({\mathbb Z}/d_0f{\mathbb Z})^*$ 
and as $\chi$ runs over $X_f^-(H)$ the $\chi'$'s run over $X_{d_0f}^-(H_{d_0})$ 
and by (\ref{M(f,H)}) and (\ref{Md0(f,H)}) we have
\begin{equation}\label{MM}
M_{d_0}(f,H) 
=M(d_0f,H_{d_0}).
\end{equation} 
The following Lemma is probably well known but we found no reference in the literature.

\begin{lemma}\label{kernel}
Let $f>2$. 
Let $H$ be a subgroup of index $m =(G:H)$ in the multiplicative group $G:=({\mathbb Z}/f{\mathbb Z})^*$. 
Then $\# X_f(H) =m$ and $H=\cap_{\chi\in X_f(H)}\ker\chi$.
Moreover, if $-1\not\in H$, then $m$ is even, $\# X_f^-(H)=m/2$ and $H=\cap_{\chi\in X_f^-(H)}\ker\chi$.
\end{lemma}

\begin{proof}
Since $X_f(H)$ is isomorphic to the group of Dirichlet characters of the abelian quotient group $G/H$, 
it is of order $m$, 
by \cite[Chapter VI, Proposition 2]{Ser}. 
Clearly, $H\subseteq \cap_{\chi\in X_f(H)}\ker\chi$. 
Conversely, take $g\not\in H$, of order $n\geq 2$ in the abelian quotient group $G/H$. 
Define a character $\chi$ of the subgroup $\langle g,H\rangle$ of $G$ generated by $g$ and $H$ 
by $\chi(g^kh) =\exp(2\pi ik/n)$, $(k,h)\in {\mathbb Z}\times H$.
It extends to a character of $G$ still denoted $\chi$, 
by \cite[Chapter VI, Proposition 1]{Ser}. 
Since $g\not\in\ker\chi$ and $\chi\in X_f(H)$ we have $g\not\in \cap_{\chi\in X_f(H)}\ker\chi$, 
i.e. $ \cap_{\chi\in X_f(H)}\ker\chi\subseteq H$.

Now, assume that $-1\not\in H$. Set $H'=\langle -1,H\rangle$, of index $m/2$ in $G$. 
Then $X_f^-(H) =X_f(H)\setminus X_f(H')$ is indeed of order $m-m/2 =m/2$, 
by the first assertion. 
Clearly, $H\subseteq \cap_{\chi\in X_f^-(H)}\ker\chi$. 
Conversely, take $g\not\in H$. 
Set $H'':=\langle g,H\rangle=\{g^kh;\ k\in {\mathbb Z},\ h\in H\}$, of index $m''$ in $G$, with $m>m''$.
If $-1=g^kh\in H''$ then clearly $\chi (g)\neq 1$ for $\chi\in X_f^-(H)$, 
hence $g\not\in \cap_{\chi\in X_f^-(H)}\ker\chi$. 
If $-1\not\in H''$ and $\chi\in X_f^-(H)\setminus X_f^-(H'')$, 
a non-empty set or cardinal $m/2-m''/2 =(H'':H)/2\geq 1$, 
then clearly $\chi (g)\neq 1$, hence $g\not\in \cap_{\chi\in X_f^-(H)}\ker\chi$. 
Therefore, $\cap_{\chi\in X_f^-(H)}\ker\chi\subseteq H$.
\end{proof}

\begin{Remark}\label{restrictiond0}
We have 
$M_{d_0}(p,H)/D_{d_0}(p,H)
=M_{d_0/q}(p,H)/D_{d_0/q}(p,H)$ 
whenever a prime $q$ dividing $d_0$ is in $\cap_{\chi\in X_p^-(H)}\ker\chi$. 
Hence, by Lemma \ref{kernel}, 
when applying (\ref{boundhKminusd0}) we may assume that no prime divisor of $d_0$ is in $H$.
\end{Remark}

\subsection{On the size of $\Pi_{d_0}(f,H)$ and $D_{d_0}(f,H)$ defined in (\ref{Pi})}\label{sizeproduct}
\begin{lemma}\label{formulaPi}
Let $H$ be a subgroup of order $d\geq 1$ of the multiplicative group $({\mathbb Z}/f{\mathbb Z})^*$, 
where $f>2$. 
Assume that $-1\not\in H$. 
Let $g$ be the order of a given prime integer $q$ in the multiplicative quotient group $({\mathbb Z}/f{\mathbb Z})^*/H$. 
Let $X_f(H)$ be the multiplicative group of the $\phi(f)/d$ Dirichlet characters modulo $f$ for which $\chi_{/H} =1$. 
Define $X_f^-(H) =\{\chi\in X_f(H);\ \chi (-1)=-1\}$, 
a set of cardinal $\phi(f)/(2d)$.
Then 
$$\Pi_{q}(f,H)
:=\prod_{\chi\in X_f^-(H)}\left (1-\frac{\chi (q)}{q}\right )
=\begin{cases}
\left (1+\frac{1}{q^{g/2}}\right )^{\frac{\phi(f)}{dg}}&\hbox{if $g$ is even and $-q^{g/2}\in H$,}\\
\left (1-\frac{1}{q^g}\right )^{\frac{\phi(f)}{2dg}}&\hbox{otherwise.}
\end{cases}$$
\end{lemma}

\begin{proof}
Let $\alpha$ be of order $g$ in an abelian group $A$ of order $n$.
Let $B=\langle\alpha\rangle$ be the cyclic group generated by $\alpha$. 
Let $\hat B$ be the group of the $g$ characters of $B$.
Then 
$P_{B}(X) :=\prod_{\chi\in\hat B} (X-\chi (\alpha)) =X^g-1$.
Now, the restriction map $\chi\in\hat A\rightarrow\chi_{/B}\in\hat B$ is surjective, 
by \cite[Proposition 1]{Ser}, and of kernel isomorphic to $\widehat{A/B}$ of order $n/g$, 
by \cite[Proposition 2]{Ser}. 
Therefore,
$P_{A}(X) 
:=\prod_{\chi\in\hat A} (X-\chi (\alpha)) 
=P_B(X)^{n/g}
=(X^g-1)^{n/g}$.\\
With $A=({\mathbb Z}/f{\mathbb Z})^*/H$ of order $n=\phi(f)/d$, 
we have $\hat A =X_f(H)$ and 
$$\prod_{\chi\in X_f(H)}(X-\chi (q))
=(X^g-1)^{\frac{\phi(f)}{dg}}.$$ 
Let $H'$ be the subgroup of order $2d$ generated by $-1$ and $H$. 
With $A'=({\mathbb Z}/f{\mathbb Z})^*/H'$ of order $n'=\phi(f)/(2d)$, 
we have 
$\hat A' =X_f(H') =X_f^+(H):=\{\chi\in X_f(H);\ \chi (-1)=+1\}$ 
and 
$$\prod_{\chi\in X_f^+(H)}(X-\chi (q))
=(X^{g'}-1)^{\frac{\phi(f)}{2dg'}},$$ 
where $q$ is of order $g'$ in $A'$.\\ 
Since $X_f^-(H) =X_f(H)\setminus X_f^+(H)$, 
it follows that 
$$\prod_{\chi\in X_f^-(H)} (X-\chi (q))
=\frac{(X^g-1)^{\frac{\phi(f)}{dg}}}{(X^{g'}-1)^{\frac{\phi(f)}{2dg'}}}.$$
Since $q^g\in H$ we have $q^g\in H'$ and $g'$ divides $g$. 
Since $q^{g'}\in H'=\{\pm h;\ h\in H\}$ we have $q^{2g'}\in H$ and $g$ divides $2g'$. 
Hence, $g=g'$ or $g =2g'$ and $g=2g'$ if and only if $g$ is even and $q^{g/2}=q^{g'}\in H'\setminus H=\{-h;\ h\in H\}$. 
The assertion follows.
\end{proof}

\begin{corollary}\label{PiKminus}
Fix $d_0>1$ square-free. 
Let $p\geq 3$ run over the prime integers that do not divide $d_0$. 
Let $H$ a subgroup of odd order $d$ of the multiplicative group $({\mathbb Z}/p{\mathbb Z})^*$.
Then,
\begin{equation}\label{behaviorPidexplicit}
D_{d_0}(p,H)\
=1+O(\omega(d_0)p^{-1/2(d-1)})
\end{equation}
where $\omega(d_0)$ stands for the number of prime divisors of $d_0$. In particular when $d =o(\log p)$, 
we have
\begin{equation}\label{behaviorPid0pH}
D_{d_0}(p,H)=1+o(1).
\end{equation}
Moreover, 
$$\Pi _{d_0}(p,\{1\})
\geq\exp\left (\frac{\log d_0}{2}F(p+1)\right ), 
\hbox{ where } 
F(x):=\frac{(x-2)\log\left (1-\frac{1}{x}\right )}{\log x}, \\\ (x>1).$$
In particular, $\Pi_{6}(p,\{1\}) \geq 2/3 \,\, \hbox{for $p\geq 5$.}$
\end{corollary}

\begin{proof}
Let $q$ be a prime divisor of $d_0$. 
Let $g$ be the order of $q$ in the multiplicative quotient group $({\mathbb Z}/p{\mathbb Z})^*/H$. 
Then 
$$\left (1-\frac{1}{q^g}\right )^{\frac{2}{g}}
\leq D_q(p,H)
=\Pi_{q}(p,H)^{\frac{4d}{p-1}} 
\leq\left (1+\frac{1}{q^{g/2}}\right )^{\frac{4}{g}},$$
by (\ref{Pi}) and Lemma \ref{formulaPi}, with $f=p$, $\phi(f)=p-1$ and $m =(p-1)/d$.
Either $q^g\equiv 1\pmod p$, in which case $q^g\geq p+1$, 
or $q^g\equiv h\pmod p$ for some $h\in\{2,\cdots,p-1\}\cap H$, 
in which case $p$ divides 
$S:=1+h+\cdots +h^{d-1}$ 
which satisfies 
$p\leq S\leq 2h^{d-1}$. 
Therefore, in both cases, we have $q^{g}\geq (p/2)^{\frac{1}{d-1}}$. Hence,
$$\log D_q(p,H) \geq \frac{2}{g} \log(1-q^{-g}) \geq \frac{2}{g}(-2\log 2)q^{-g} \geq -4(\log 2)(p/2)^{-1/(d-1)}$$
where we used for $x=q^{-g}$ the fact that $\log(1-x) \geq -2(\log 2)x$ in $\left[0,1/2\right]$.
$$ D_q(p,H) \geq 1-4(\log 2)(p/2)^{-1/(d-1)}$$ where we used the fact that $e^{-x} \geq 1-x$. 
Therefore we have, 
$$D_{d_0}(p,H) =\prod_{q \mid d_0}D_q(p,H) \geq 1-4(\log 2)\omega(d_0)\left(\frac{p}{2}\right)^{-1/(d-1)}$$ 
where we used the inequality $(1-x)^n \geq 1-nx$ for $x \leq 1$ and $n\in \mathbb{N}$. 
A similar reasoning gives an explicit upper bound $D_{d_0}(p,H) \leq 1+ c\omega(d_0)p^{-1/2(d-1)}$ 
for some constant $c>0$. 
Therefore, we do get \eqref{behaviorPidexplicit}. 
Finally, $p^{1/(d-1)}$ tends to infinity in the range $d =o(\log p)$
and (\ref{behaviorPid0pH}) follows. 

Notice that if $p=2^d-1$ runs over the Mersenne primes and $H=\langle 2\rangle$, 
we have $d=O(\log p)$ but $D_2(p,H) =\left (1-\frac{1}{2}\right )^2$ does not satisfy (\ref{behaviorPid0pH}).

Now, assume that $H=\{1\}$. 
Then, $K={\mathbb Q}(\zeta_p)$ and 
$q^g\geq p+1$. 
Hence,
$$\Pi_q(p,\{1\})
\geq\left (1-\frac{1}{p+1}\right )^{\frac{p-1}{2g}} 
\geq\left (1-\frac{1}{p+1}\right )^{\frac{(p-1)\log q}{2\log (p+1)}} 
=\exp\left (\frac{\log q}{2}F(p+1)\right ).$$
The desired lower bound easily follows. 
\end{proof}

\section{Dedekind sums and mean square values of $L$-functions}\label{Dedekind}
\subsection{Dedekind sums and Dedekind-Rademacher sums}
The Dedekind sums is the rational number defined by
\begin{equation}\label{defscd}
s(c,d)
={1\over 4d}\sum_{n=1}^{\vert d\vert -1}\cot\left ({\pi n\over d}\right )\cot\left ({\pi nc\over d}\right )
\ \ \ \ \ (c\in {\mathbb Z},\ d\in {\mathbb Z}\setminus\{0\},\ \gcd (c,d)=1),
\end{equation}
with the convention $s(c,-1)=s(c,1)=0$ for $c\in {\mathbb Z}$
(see \cite{Apo} or \cite{RG} where it is however assumed that $d>1$).
It depends only on $c$ mod $\vert d\vert$ 
and $c\mapsto s(c,d)$ can therefore be seen as a mapping 
from $({\mathbb Z}/\vert d\vert {\mathbb Z})^*$ to ${\mathbb Q}$.
Notice that 
\begin{equation}\label{cc*}
s(c^*,d)=s(c,d)
\text{ whenever }
cc^*\equiv 1\pmod d
\end{equation} 
(make the change of variables $n\mapsto nc$ in $s(c^*,d)$). 
Recall the reciprocity law for Dedekind sums 
\begin{equation}\label{scddc}
s(c,d)+s(d,c) 
={c^2+d^2-3\vert cd\vert +1\over 12cd},
\ \ \ \ \ (c,d\in {\mathbb Z}\setminus\{0\},\ \gcd(c,d)=1).
\end{equation}
In particular, 
\begin{equation}\label{s1d}
s(1,d)
=\frac{d^2-3\vert d\vert +2}{12d}
\text{ and } 
s(2,d)
=\frac{d^2-6\vert d\vert +5}{24d}
\ \ \ \ \ (d\in {\mathbb Z}\setminus\{0\}).
\end{equation}
For $b,c\in {\mathbb Z},\ d\in {\mathbb Z}\setminus\{-1,0,1\}$ such that $\gcd (b,d) =\gcd (c,d)=1$, 
the Dedekind-Rademacher sum is the rational number defined by
$$s(b,c,d)
={1\over 4d}\sum_{n=1}^{\vert d\vert -1}\cot\left ({\pi nb\over d}\right )\cot\left ({\pi nc\over d}\right ),$$ 
with the convention $s(b,c,-1)=s(b,c,1)=0$ for $b,c\in {\mathbb Z}$.
Hence, $s(c,d) =s(1,c,d)$, 
if $\alpha\in ({\mathbb Z}/\vert d\vert {\mathbb Z})^*$ 
is represented as $\alpha =b/c$ with $\gcd(b,d)=\gcd(c,d) =1$, 
then $s(\alpha,d) =s(b,c,d)$, 
and 
\begin{equation}\label{bcd=abacd}
\hbox{$s(b,c,d) =s(ab,ac,d)$ for any $a\in {\mathbb Z}$ with $\gcd (a,d)=1$.}
\end{equation}
For $\gcd( b,c) =\gcd (c,d) =\gcd (d,b) =1$ 
we have a reciprocity law for Dedekind-Rademacher sums 
(see \cite{Rad} or \cite{BR}):
\begin{equation}\label{sbcddcb}
s(b,c,d)+s(d,b,c)+s(c,d,b) ={b^2+c^2+d^2-3\vert bcd\vert\over 12bcd}.
\end{equation}
The Cauchy-Schwarz inequality and (\ref{s1d}) yield 
\begin{equation}\label{boundscd}
\vert s(c,d)\vert
\leq s(1,\vert d\vert)
\leq\vert d\vert/12
\hbox { and }
\vert s(b,c,d)\vert
\leq s(1,\vert d\vert)
\leq\vert d\vert/12.
\end{equation}

\subsection{Non trivial bounds on Dedekind sums}\label{boundsDedekind}
In this section we will use the alternative definition of the Dedekind sums given by
$$s(c,d)
=\sum_{a=1}^{d-1} 
\left(\left(\frac{a}{d}\right)\right) \left(\left(\frac{ac}{d}\right)\right) 
\ \ \ \ \ (c\in {\mathbb Z},\ d\geq 1,\ \gcd (c,d)=1)$$ 
where $ \left(\left(\right)\right):\mathbb{R} \rightarrow \mathbb{R}$ stands for the sawtooth function defined by 
$$\left(\left(x\right)\right)
:=
\begin{cases}
x-\lfloor x\rfloor - 1/2& \text{ if } x \in \mathbb{R} \backslash \mathbb{Z},\\
0 &\text{ if } x \in \mathbb{Z}.
\end{cases}
$$
In order to prove Theorem \ref{asympd0}, 
we need general bounds on Dedekind sums depending on the multiplicative order of the argument. 
This is a new type of bounds for Dedekind sums 
and the following result that improves upon \eqref{boundscd} when the order is $o\left(\frac{\log p}{\log\log p}\right)$ 
might be of independent interest (see also Conjecture \ref{conjDedekind} for further discussions).

\noindent\frame{\vbox{
\begin{theorem}\label{indivbound}
Let $p>1$ be a prime integer 
and assume that $h$ has odd order $k\geq 3$ in the multiplicative group $({\mathbb Z}/p{\mathbb Z})^*$. 
We have
$$\vert s(h,p)\vert \ll (\log p)^2 p^{1-\frac{1}{\phi(k)}}.$$
\end{theorem}}}
 
\begin{Remark}
Let us notice that by a result of Vardi \cite{Var}, for any function $f$ such that $\lim_{n\to +\infty} f(n)=+\infty$ 
we have $s(c,d) \ll f(d)\log d$ for almost all $(c,d)$ with $\gcd (c,d)=1$. 
However Dedekind sums take also very large values 
(see for instance \cite{CEK,Gir03} for more information). \end{Remark}

Our proof builds from ideas of the proof of \cite[Theorem $4.1$]{LMQJM} 
where some tools from equidistribution theory and the theory of pseudo-random generators were used. 
We refer for more information to \cite{Korobov}, \cite{niederreiter1977pseudo} 
or the book of Konyagin and Shparlinski \cite[Chapter $12$]{igorkonya} 
(see \cite[Section $4$]{LMQJM} for more details and references). 
Let us recall some notations. 
For any fixed integer $s$, we consider the $s$-dimensional cube $I_s=\left[0,1\right]^s$ 
equipped with its $s$-dimensional Lebesgue measure $\lambda_s$. 
We denote by $\mathcal{B}$ the set of rectangular boxes of the form 
$$\prod_{i=1}^{s}[\alpha_i,\beta_i)
=\left\{x\in I_s, \alpha_i\leq x_i <\beta_i\right\}$$ 
where $0\leq \alpha_i<\beta_i\leq 1.$ If $S$ is a finite subset of $I^s$, 
we define the discrepancy $D(S)$ by 
$$D(S)
=\sup_{B \in \mathcal{B}}\left\vert \frac{\# (B\cap S)}{\# S}-\lambda_s(B)\right\vert.$$ 
Let us introduce the following set of points: 
$$S_{h,p}
=\left\{\left(\frac{x}{p},\frac{x h}{p}\right) \in I_2, x \bmod p\right\}.$$ 
For good choice of $h$, the points are equidistributed and we expect for ``nice'' functions $f$
$$\lim_{p\to \infty} \frac{1}{p}\sum_{x \bmod p} f\left(\frac{x}{p},\frac{hx}{p}\right) = \int_{I_2} f(x,y)dx dy.$$

\begin{lemma}\label{discrepancy}
For any $h$ of odd order $k\geq 3$ we have the following discrepancy bound
$$D(S_{h,p}) 
\leq (\log p)^{2}p^{-1/\phi(k)}.$$
\end{lemma}

\begin{proof}
It follows from the proof of \cite[Theorem 4.1]{LMQJM} where the bound was obtained 
as a consequence of Erd\H{o}s-Turan inequality and tools from pseudo random generators theory.
\end{proof}

\subsubsection{Proof of Theorem \ref{indivbound}} 
Observe that $$s(h,p)= \sum_{x \bmod p} f\left(\frac{x}{p},\frac{hx}{p}\right)$$ where $f(x,y)= ((x))((y))$.
By Koksma-Hlawka inequality \cite[Theorem $1.14$]{tichy} we have 
$$\left\vert
\frac{1}{p} \sum_{ x \bmod p} f\left(\frac{x}{p},\frac{x h}{p}\right) 
-\int_{I_{2}} f(u,v)dudv
\right\vert 
\leq V(f)D(S_{h,p})$$ where $V(f)$ is the Hardy-Krause variation of $f$. Moreover we have 
$$\int_{I_{2}} f(u,v)dudv 
=0.$$ 
The readers can easily convince themselves that $V(f) \ll 1$. Hence the result follows from Lemma \ref{discrepancy}.

\begin{Remark}\label{compositebound}
The same method used to bound the discrepancy leads to a similar bound for composite $f$. 
Indeed for $h \in ({\mathbb Z}/f{\mathbb Z})^*$ of order $k\geq 3$, 
we have
$s(h,f)= O\left((\log f)^2 f/E(f)\right)$ with $E(f)=\max\{P^{+}(f)^{1/\phi(k^*)},\textrm{rad}(f)^{1/k}\}$ 
where $P^{+}(f)$ is the largest prime factor of $f$, $k^*$ is the order of $h$ modulo $P^+(f)$ 
and $\displaystyle{rad(f)=\prod_{\ell \mid f \atop \ell \textrm{prime}}\ell}$ is the radical of $f$. If $f=h^3-1$ is squarefree, 
then we have $E(f) = f^{1/3}$ and $s(h,f)=O\left((\log f)^2 f^{2/3}\right)$ 
which is close to the truth by a logarithmic factor 
(see Remark \ref{dedekindtwosizes}).
\end{Remark}

For $\gcd(b,p)=\gcd(c,p)=1$ we recall the other definition of Dedekind-Rademacher sums
$$s(b,c,p)= \sum_{a=1}^{p-1} \left(\left(\frac{ab}{p}\right)\right) \left(\left(\frac{ac}{p}\right)\right).$$ 
A similar argument as in the proof of Theorem \ref{indivbound} leads to a bound on these generalized sums:
\begin{theorem}
\label{indivboundgen}
Let $q_1$, $q_2$ and $k\geq 3$ be given natural integers. 
Let $p$ run over the primes 
and $h$ over the elements of order $k$ in the multiplicative group $({\mathbb Z}/p{\mathbb Z})^*$. 
Then, we have
$$\vert s(q_1,q_2 h,p)\vert \ll (\log p)^2 p^{1-\frac{1}{\phi(k)}}.$$ 
\end{theorem}
 
\begin{proof}
The proof follows exactly the same lines as the proof of Theorem \ref{indivbound} 
except for the fact that the function $f$ is replaced by the function $g(x,y)= ((q_1 x))((q_2 y))$. 
Hence we have $$s(q_1,q_2 h,p)= g\left(\frac{x}{p},\frac{hx}{p}\right)$$ and by symmetry we remark that
$$\int_{I_{2}} g(u,v)dudv 
=0.$$ 
Again $V(g) \ll 1$ and the result follows from Lemma \ref{discrepancy} and Koksma-Hlawka inequality.
\end{proof}

\subsection{Twisted second moment of $L$- functions and Dedekind sums}\label{twistedsection}
We illustrate the link between Dedekind sums and twisted moments of $L$- functions 
by first proving Theorem \ref{asympd0} in the case $H=\{1\}$ with a stronger error term. 
For any integers $q_1,q_2 \geq 1$ and any prime $p\geq 3$, we define the twisted moment
\begin{equation}\label{twistedmomentrivial} 
M_{q_1,q_2}(p)
:=\frac{2}{\phi(p)} \sum_{\chi \in X_p^-}\chi(q_1)\overline{\chi}(q_2)\vert L(1,\chi)\vert^2.
\end{equation}

The following formula (see \cite[Proposition 1]{LouCMB36/37}) will help us to relate $L$- functions to Dedekind sums:
\begin{equation}\label{formulaL1X}
L(1,\chi)
=\frac{\pi}{2f}\sum_{a=1}^{f-1}\chi (a)\cot\left (\frac{\pi a}{f}\right )
\ \ \ \ \ (\chi\in X_f^-).\end{equation}

\begin{theorem}\label{twistedmoment}
Let $q_1$ and $q_2$ be given coprime integers. Then when $p$ goes to infinity
$$ M_{q_1,q_2}(p)= \frac{\pi^2}{6q_1q_2}+O_{q_1,q_2}(1/p). $$
\end{theorem}

\begin{Remark}
 It is worth to notice that in the case $q_2=1$, explicit formulas are known by \cite[Theorem $4$]{LouBKMS52} 
(see also \cite{Lee17}). 
This also gives a new and simpler proof of \cite[Theorem $1.1$]{Lee19} in a special case.
\end{Remark}
 
\begin{proof}
Let us define
$$\epsilon(a,b)
:=\frac{2}{\phi(p)}\sum_{\chi \in X_p^-}\chi(a)\overline{\chi}(b)
=\begin{cases} 
1&\text{ if } p\nmid ab \text{ and } a=b \bmod p, \\
-1&\text{ if } p\nmid ab \text{ and } a=-b \bmod p, \\
0 &\text{ otherwise.}
\end{cases}$$ 
For $p$ large enough, we have $\gcd(q_1,p)=\gcd(q_2,p)=1$. Hnece, using orthogonality relations and \eqref{formulaL1X} we arrive at 
\begin{align*} 
M_{q_1,q_2}(p)
&=\frac{\pi^2}{4p^2} 
\sum_{a=1}^{p-1} 
\sum_{b=1}^{p-1}\epsilon(q_1 a,q_2 b)\cot\left(\frac{\pi a}{p}\right)\cot\left(\frac{\pi b}{p}\right) \\
& = \frac{\pi^2}{2p^2} \sum_{a=1}^{p-1}\cot\left(\frac{\pi q_1a}{p}\right)\cot\left(\frac{\pi q_2 a}{p}\right)
= \frac{2\pi^2}{p}s(q_1,q_2,p).
\end{align*} 
When $q_1$ and $q_2$ are fixed coprime integers and $p$ goes to infinity, 
we infer from \eqref{sbcddcb} and \eqref{boundscd} that 
$$s(q_1,q_2,p)= \frac{p}{12q_1 q_2}+O(1). $$
The result follows immediatly.
\end{proof}

\begin{corollary}\label{twistednoncoprime}
Let $q_1$ and $q_2$ be given natural integers. Then when $p$ goes to infinity
$$ M_{q_1,q_2}(p
)=\frac{\pi^2}{6}\frac{ \gcd(q_1,q_2)^2}{q_1q_2}+O_{q_1,q_2}(1/p). $$
\end{corollary} 

\begin{proof} 
Let $\delta= \gcd(q_1,q_2)$. 
We clearly have $M_{q_1,q_2}(p)=M_{q_1/\delta,q_2/\delta}(p)$ 
and the result follows from Theorem \ref{twistedmoment}.
\end{proof}

The proof of Theorem \ref{asympd0} in the case of the trivial subgroup follows easily.

\begin{corollary}\label{theotrivialcase}
Let $d_0$ be a given square-free integer. 
When $p$ goes to infinity, we have the following asymptotic formula 
$$M_{d_0}(p,\{1\})
 =\frac{\pi^2}{6}\prod_{q \mid d_0} \left(1-\frac{1}{q^2}\right)+ O(1/p).$$
\end{corollary}
 
\begin{proof} 
For $\chi$ modulo $p$, let $\chi'$ be the character modulo $d_0p$ induced by $\chi$. 
By \eqref{L1XXprime} and Corollary \ref{twistednoncoprime} we have 
\begin{align*} M_{d_0}(p,\{1\})& =\frac{2}{\# X_p^-} \sum_{\chi \in X_p^-} |L(1,\chi')|^2 
=\sum_{\delta_1 \mid d_0} 
\sum_{\delta_2 \mid d_0} 
\frac{\mu(\delta_1)}{\delta_1}\frac{\mu(\delta_2)}{\delta_2} M_{\delta_1,\delta_2}(p) \\
&=\frac{\pi^2}{6}\sum_{\delta_1 \mid d_0} 
\sum_{\delta_2 \mid d_0} 
\frac{\mu(\delta_1)}{\delta_1^2}
\frac{\mu(\delta_2)}{\delta_2^2}\gcd(\delta_1,\delta_2)^2 
+ O(1/p) \\
&=\frac{\pi^2}{6} \prod_{q \mid d_0} \left(1-\frac{1}{q^2}\right)+O(1/p).
\end{align*}
\end{proof}

\subsection{An interesting link with sums of maxima}
Before turning to the general case of Theorem \ref{asympd0}, 
we explain how to use Theorem \ref{twistedmoment} to estimate the seemingly
innocuous sum\footnote{In \cite{Sun} the author uses lattice point interpretation to study sums with a similar flavour.} 
defined for any integers $q_1,q_2 \geq 1$ by
$$\text{Ma}_{q_1,q_2,p}:=\sum_{x \bmod p}\max(q_1 x,q_2 x)$$
where here and below $q_1 x, q_2 x$ denote the representatives modulo $p$ taken in $\left[1,p\right]$.

\noindent\frame{\vbox{
\begin{theorem}\label{summax} 
Let $q_1$ and $q_2$ be natural integers such that $q_1 \neq q_2$. Then we have the following asymptotic formula
$$\text{Ma}_{q_1,q_2,p}
= p^2\left(\frac{2}{3}-\frac{\gcd(q_1,q_2)^2}{12q_1q_2}\right)(1+o(1)).$$ 
\end{theorem}
}}

\begin{Remark} 
In the special case $q_1=1$, we are able to evaluate the sum directly without the need of Dedekind sums and $L$- functions. 
However, we could not prove Theorem \ref{summax} in the general case using elementary counting methods. 
\end{Remark}

\begin{Remark}
Let us notice that $\int_{0}^{1}\int_{0}^{1} \max(x,y) dx dy = 2/3$. 
Hence using the same method as in Section \ref{boundsDedekind}, 
we can show that if the points $\left(\left\{\frac{x}{p}\right\},\left\{\frac{qx}{p}\right\}\right)$ 
are equidistributed in the square $[0,1]^2$ 
then 
$$\sum_{x \bmod p} \max(x,qx) \sim \frac{2}{3}p^2.$$ 
For $q$ fixed and $p \rightarrow +\infty$, the points are not equidistributed in the square
and we see that the correcting factor $\frac{\gcd(q_1,q_2)^2}{12q_1q_2}$ from equidistribution 
is related to the Dedekind sum $s(q_1,q_2,p)$. 
\end{Remark}

We need the following result of \cite[Theorem $2.1$]{LMQJM}:

\begin{proposition}\label{relation}
Let $\chi$ be a primitive Dirichlet character modulo $f>2$, its conductor. 
Set $\displaystyle{S(k,\chi) =\sum_{l=0}^k\chi (l)}$.
Then 
$$\sum_{k=1}^{f-1}\vert S(k,\chi)\vert^2 
=\frac{f^2}{12}\prod_{p\mid f}\left (1-\frac{1}{p^2}\right )
+a_\chi\frac{f^2}{\pi^2}\vert L(1,\chi)\vert^2, 
\hbox { where }
a_\chi
:=\begin{cases}
0&\hbox{if $\chi (-1)=+1$,}\\
1&\hbox{if $\chi (-1)=-1$.}
\end{cases}$$
\end{proposition}

\subsubsection{Proof of Theorem \ref{summax}}\label{sectionsummax}
We follow a strategy similar to the proof of \cite[Corollary $2.2$]{LMQJM}. 
We denote by $\chi_0$ the trivial character.
Using Proposition \ref{relation} and recalling the definition \eqref{twistedmomentrivial} we arrive at:
$$ \sum_{\chi \in X_p \backslash \chi_0}\chi(q_1)\overline{\chi}(q_2)\sum_{k=1}^{p-1}\vert S(k,\chi)\vert^2 
= \sum_{\chi \in X_p\backslash \chi_0}\chi(q_1)\overline{\chi}(q_2)\frac{p^2-1}{12}
+\frac{p^3}{2\pi^2}M_{q_1,q_2}(p). $$ 
Adding the contribution of the trivial character 
$$ \chi_0(q_1)\overline{\chi_0}(q_2)\sum_{k=1}^{p-1}\left\vert\sum_{l=1}^k 1\right\vert^2
=\sum_{k=1}^{p-1} k^2 
=\frac{(p-1)p(2p-1)}{6},$$ 
we obtain
\begin{align}\label{shift} 
\sum_{\chi \in X_p}\chi(q_1)\overline{\chi}(q_2)\sum_{k=1}^{p-1}\vert S(k,\chi)\vert^2 
=&\sum_{\chi \in X_p}\chi(q_1)\overline{\chi}(q_2)\frac{p^2-1}{12} + \frac{(p-1)p(2p-1)}{6} \nonumber \\ 
&+\frac{p^3}{2\pi^2}M_{q_1,q_2}(p) +O(p^2). 
\end{align} 
For sufficiently large $p$, using the fact that $q_1 \neq q_2 \bmod p$ and the orthogonality relations, we have 
$$\sum_{\chi \in X_p}\chi(q_1)\overline{\chi}(q_2)\frac{p^2-1}{12}=0.$$

We now follow the method used in the proof of \cite[Theorem $4.1$]{LMQJM} (see also \cite{Elma})
with some needed changes to treat the left hand side of \eqref{shift}. Again by orthogonality, we obtain
\begin{align*} \sum_{\chi \in X_p}\chi(q_1)\overline{\chi}(q_2)\sum_{k=1}^{p-1}\vert S(k,\chi)\vert^2
=\sum_{\chi \in X_p}\chi(q_1)\overline{\chi}(q_2)\sum_{k=1}^{p-1}\left\vert\sum_{l=1}^k\chi (l)\right\vert^2 \\
=\sum_{\chi \in X_p}\sum_{k=1}^{p-1}\sum_{1\leq l_1,l_2\leq k}
\chi(q_1 l_1)\overline{\chi(q_2 l_2)} =(p-1)^2{\mathcal {A}}(q_1,q_2,p),
\end{align*}
where 
\begin{equation*}
{\cal A}(q_1,q_2,p)
=\frac{1}{p-1}
\sum_{N=1}^{p-1}\left (\sum_{1 \leq n_1,n_2 \leq N \atop q_1 n_1=q_2 n_2 \bmod p} 1\right ).
\end{equation*} 
Changing the order of summation and making the change of variables $n_1=q_2m_1$ we arrive at 
$$(p-1){\cal A}(q_1,q_2,p)
= \sum_{1\leq m_1 \leq p}(p- \max(q_1 m_1,q_2 m_1))
=p^2
- \sum_{x \bmod p}\max(q_1 x,q_2 x).$$ 
By symmetry, injecting this into \eqref{shift}, we arrive at 
\begin{align}\label{formulacp}
p^3 -p\sum_{x \bmod p}\max(q_1 x,q_2 x) 
=\frac{(p-1)p(2p-1)}{6}+ \frac{p^3}{2\pi^2}M_{q_1,q_2}(p) +o(p^3).
\end{align}
Hence comparing the terms of order $p^3$ in the above formula \eqref{formulacp} 
and using Corollary \ref{twistednoncoprime}, 
we have
$$\sum_{x \bmod p}\max(q_1 x,q_2 x) 
=c_{q_1,q_2}(p^2 +o(p^2))$$ 
where 
$$1-c_{q_1,q_2}
=\frac{1}{3}+\frac{1}{12}\frac{\gcd(q_1,q_2)^2}{q_1 q_2}.$$ 
This concludes the proof. \\
 
We know turn to the general case of Theorem 1.1. 
Let $d_0$ be a given square-free integer such that $\gcd(d_0,p)=1$. 
For $\chi$ modulo $p$, let $\chi'$ be the character modulo $d_0p$ induced by $\chi$. 
Recall that we want to show for $H$ a subgroup of $\left(\mathbb{Z}/p\mathbb{Z}\right)^*$ 
of odd order $d \ll \frac{\log p}{\log \log p}$ that
\begin{align*}
M_{d_0}(p,H)
&=\frac{1}{\# X_p^-(H)} \sum_{\chi \in X_p^-(H)} |L(1,\chi')|^2 
=(1+o(1))\frac{\pi^2}{6}\prod_{q \mid d_0}\left(1-\frac{1}{q^2}\right).
\end{align*}

\subsection{Twisted average of $L$- functions over subgroups}
For any integers $q_1,q_2 \geq 1$ and any prime $p\geq 3$, we define
$$M_{q_1,q_2}(p,H)
:=\frac{1}{\# X_p^-(H)} \sum_{\chi \in X_p^-(H)} \chi(q_1)\overline{\chi}(q_2)|L(1,\chi)|^2.$$
Our main result is the following:

\begin{theorem}\label{MTtwist} 
Let $q_1$ and $q_2$ be given coprime integers. 
When $H$ runs over the subgroups of $\left(\mathbb{Z}/p\mathbb{Z}\right)^*$ of odd order $d$, 
we have the following asymptotic formula
$$M_{q_1,q_2}(p,H) 
=\frac{\pi^2}{6q_1 q_2} + O\left( d (\log p)^2p^{-\frac{1}{\phi(d)}} \right).$$
\end{theorem}

\begin{proof} 
The proof follows the same lines as the proof of Theorem \ref{twistedmoment}. 
Let us define
$$\epsilon_H(a,b)
:=\frac{1}{\# X_p^-(H)}\sum_{\chi \in X_p^-(H)}\chi(a)\overline{\chi}(b)
=\begin{cases} 
1 &\text{ if } p\nmid ab \text{ and } a \in bH, \\
-1 &\text{ if } p\nmid ab \text{ and } a \in -bH, \\
0 &\text{ otherwise.}
\end{cases}$$ 
Hence we obtain similarly
\begin{align*} 
M_{q_1,q_2}(p,H)
&=\frac{\pi^2}{4p^2} 
\sum_{a=1}^{p-1} 
\sum_{b=1}^{p-1}\epsilon_H(q_1 a,q_2 b)\cot\left(\frac{\pi a}{p}\right)\cot\left(\frac{\pi b}{p}\right) \\
& = \frac{\pi^2}{2p^2}\sum_{h\in H} \sum_{a=1}^{p-1}\cot\left(\frac{\pi q_1a}{p}\right)\cot\left(\frac{\pi q_2h a}{p}\right) \\
& = \frac{2\pi^2}{p}s(q_1,q_2,p)+ O\left(p^{-1}\sum_{1\neq h\in H} s(q_1,q_2h,p)\right) \\
& = \frac{\pi^2}{6q_1q_2} + O(1/p) + O\left( \vert H\vert (\log p)^2p^{-\frac{1}{\phi(d)}} \right) \\
& =\frac{\pi^2}{6q_1q_2} + O\left( d(\log p)^2p^{-\frac{1}{\phi(d)}} \right),
\end{align*} 
where we used Theorem \ref{indivboundgen} in the last line 
and noticed that $\phi(k)$ divides $\phi(d)$ whenever $k$ divides $d$. 
\end{proof}
 
\begin{Remark}
The error term is negligible as soon as $d\leq \frac{\log p}{3(\log \log p)}$.
\end{Remark}

\noindent\frame{\vbox{ 
\begin{corollary}\label{twistednoncoprimeH}
Let $q_1$ and $ q_2$ be given integers. 
When $H$ runs over the subgroups of $\left(\mathbb{Z}/p\mathbb{Z}\right)^*$ of odd order $d$, 
we have the following asymptotic formula
$$M_{q_1,q_2}(p,H) 
= \frac{\pi^2}{6}\frac{\gcd(q_1,q_2)^2}{q_1 q_2} 
+ O\left( d (\log p)^2p^{-\frac{1}{\phi(d)}} \right).$$
\end{corollary}
}}

\subsection{Proof of Theorem \ref{asympd0}}
As in the proof of Corollary \ref{theotrivialcase} and using Corollary \ref{twistednoncoprimeH}
\begin{align*} 
M_{d_0}(p,H)
& =\frac{1}{\# X_p^-(H)} \sum_{\chi \in X_p^-(H)} |L(1,\chi')|^2 
=\sum_{\delta_1 \mid d_0} 
\sum_{\delta_2 \mid d_0} 
\frac{\mu(\delta_1)}{\delta_1}\frac{\mu(\delta_2)}{\delta_2} M_{\delta_1,\delta_2}(p,H) \\
&= \frac{\pi^2}{6}
\sum_{\delta_1 \mid d_0} 
\sum_{\delta_2 \mid d_0} 
\frac{\mu(\delta_1)}{\delta_1^2}\frac{\mu(\delta_2)}{\delta_2^2}\gcd(\delta_1,\delta_2)^2 
+ O\left( d (\log p)^2p^{-\frac{1}{\phi(d)}} \right) \\
& = \frac{\pi^2}{6} \prod_{q \mid d_0} \left(1-\frac{1}{q^2}\right) 
+ O\left( d (\log p)^2p^{-\frac{1}{\phi(d)}} \right) = (1+o(1))\frac{\pi^2}{6} \prod_{q \mid d_0}\left(1-\frac{1}{q^2}\right) 
\end{align*} 
using the condition on $d$.

\section{Explicit formulas for $M_{d_0}(f,H)$} 
Recall that by \eqref{formulaL1X}
\begin{equation*}
L(1,\chi)
=\frac{\pi}{2f}\sum_{a=1}^{f-1}\chi (a)\cot\left (\frac{\pi a}{f}\right )
\ \ \ \ \ (\chi\in X_f^-).
\end{equation*} 
Hence using the definition of Dedekind sums we obtain (see \cite[Proof of Theorem 2]{LouBPASM64})
\begin{equation}\label{formulaM(f,H)}
M(f,H)
={2\pi^2\over f}\sum_{\delta\mid f}\frac{\mu (\delta)}{\delta}\sum_{h\in H}s(h,f/\delta).
\end{equation}

\subsection{A formula for $M_{d_0}(f,\{1\})$ for $d_0=1,2,3,6$}\label{Casetrivialgroup}
The first consequence of (\ref{formulaM(f,H)}) is a short proof of \cite[Th\'eor\`emes 2 and 3]{LouCMB36/37} 
by taking $H=\{1\}$, the trivial subgroup of the multiplicative group $({\mathbb Z}/f{\mathbb Z}^*)$.
Indeed, (\ref{formulaM(f,H)}) and (\ref{s1d}) give 
$$M(f,\{1\})
={2\pi^2\over f}\sum_{\delta\mid f}\frac{\mu (\delta)}{\delta}s(1,f/\delta)
=\frac{\pi^2}{6}\sum_{\delta\mid f}\mu (\delta)\left (\frac{1}{\delta^2}-\frac{3}{\delta f}+\frac{2}{f^2}\right ).$$
The arithmetic functions $f\mapsto\sum_{\delta\mid f}\mu (\delta)\delta^k$ being multiplicative, 
we obtain (see also \cite{Qi})
\begin{equation}\label{M(f,1)}
M(f,\{1\})
={\pi^2\over 6}
\times\left\{
\prod_{q\mid f}\left (1-{1\over q^2}\right )-{3\over f}\prod_{q\mid f}\left (1-{1\over q}\right )
\right\}
\ \ \ \ \ \hbox{($f>2$).}
\end{equation}
Now, it is clear by \eqref{MM} that for $d_0$ odd and square-free and $f$ odd we have
\begin{equation*}
M_{2d_0}(f,\{1\})
=M_{d_0}(2f,\{1\}).
\end{equation*}
Hence, on applying (\ref{M(f,1)}) to $2f$ instead of $f$ we therefore obtain 
\begin{equation*}
M_2(f,\{1\})
={\pi^2\over 8}
\times\left\{
\prod_{q\mid f}\left (1-{1\over q^2}\right )-{1\over f}\prod_{q\mid f}\left (1-{1\over q}\right )
\right\}
\ \ \ \ \ \hbox{($f>2$ odd)}.
\end{equation*}

For $d_0\in\{3,6\}$, the following explicit formula holds true for any $f$ coprime with $d_0$. 
It generalizes \cite[Th\'eor\`eme 4]{LouCMB36/37} to composite moduli

\noindent\frame{\vbox{
\begin{theorem}\label{M3M6}
Let $d_0>2$ be a given square-free integer. 
Set 
$$\kappa_{d_0}
:=\frac{\pi^2}{6}
\prod_{q\mid d_0}\left (1-\frac{1}{q^2}\right )
\hbox{ and }
c
:=3\prod_{q\mid d_0}\frac{q-1}{q+1}.$$
For $n\in {\mathbb Z}$, 
set $\varepsilon(n) =+1$ if $n\equiv +1\pmod {d_0}$ 
and $\varepsilon(n) =-1$ if $n\equiv -1\pmod {d_0}$.\\
Then for $f>2$ such that all its prime divisors $q$ satisfy $q\equiv\pm 1\pmod {d_0}$ we have 
$$M_{d_0}^{}(f,\{1\})
=\kappa_{d_0}
\times\left\{
\prod_{q\mid f}\left (1-\frac{1}{q^2}\right )
-\frac{c}{f}\prod_{q\mid f}\left (1-\frac{1}{q}\right )
+\varepsilon(f)
\frac{c-1}{f}\prod_{q\mid f}\left (1-\frac{\varepsilon(q)}{q}\right )
\right\}.$$
In particular, 
for $f>2$ such that all its prime divisors $q$ satisfy $q\equiv 1\pmod {d_0}$ we have 
$$M_{d_0}^{}(f,\{1\})
=\kappa_{d_0}
\times\left\{
\prod_{q\mid f}\left (1-\frac{1}{q^2}\right )
-\frac{1}{f}\prod_{q\mid f}\left (1-\frac{1}{q}\right )
\right\}.$$
\end{theorem}
}}

\begin{proof}
With the notation of \cite[Lemma 2]{LouPMDebr78} we have 
$M_{d_0}(f,\{1\}))
=4\pi^2S(d_0,f)$.
Hence, by \cite[Lemmas 3 and 6]{LouPMDebr78} we have 
$$M_{d_0}(f,\{1\})
=\frac{\pi^2}{6}\prod_{q\mid d_0f}\left (1-\frac{1}{q^2}\right )
-\frac{\pi^2}{2}\frac{\phi(d_0)^2\phi(f)}{d_0^2f^2}
+\frac{\pi^2}{2d_0^2f}\sum_{d\mid f}\frac{\mu(d)}{d}A(d_0,f/d),$$
where the $A(d_0,f/d)$'s are rational numbers 
such that $A(d_0,f/d) =\varepsilon A(d_0,1)$ if $f/d\equiv\varepsilon\pmod {d_0}$ with $\varepsilon\in\{\pm 1\}$, 
see (\ref{Ad0f}). 
If all the prime divisors $q$ of $f$ satisfy $q\equiv\pm 1\pmod {d_0}$ 
then $f/d\equiv\varepsilon (f/d)\pmod{d_0}$ 
and 
$A(d_0,f/d)
=\varepsilon (f/d)A(d_0,1)
=\varepsilon(f)A(d_0,1)\varepsilon(d)$ 
and 
$$\sum_{d\mid f}\frac{\mu(d)}{d}A(d_0,f/d)
=\varepsilon(f)A(d_0,1)\prod_{q\mid f}\left (1-\frac{\varepsilon(q)}{q}\right ).$$
Hence we finally get 
$$M_{d_0}(f,\{1\})
=\frac{\pi^2}{6}\prod_{q\mid d_0f}\left (1-\frac{1}{q^2}\right )
-\frac{\pi^2}{2}\frac{\phi(d_0)^2\phi(f)}{d_0^2f^2}
+\frac{\pi^2}{2d_0^2f}\varepsilon(f)A(d_0,1)\prod_{q\mid f}\left (1-\frac{\varepsilon(q)}{q}\right ).$$ 
The desired formula for $M_{d_0}(f,\{1\})$ follows by using the explicit formula 
$$A(d_0,1)
=\phi(d_0)^2
-\frac{d_0^2}{3}\prod_{q\mid d_0}\left (1-\frac{1}{q^2}\right )$$ 
given in \cite[Lemma 6]{LouPMDebr78}.
\end{proof}

\subsection{A formula for $M(p,H)$}
The second immediate consequence of (\ref{formulaM(f,H)}) and (\ref{s1d}) is:
 
\noindent\frame{\vbox{
\begin{proposition}
For $f>2$ and $H$ a subgroup of the multiplicative group $({\mathbb Z}/f{\mathbb Z})^*$, 
set 
\begin{equation}\label{N(f,H)}
S'(H,f)=\sum_{1\neq h\in H}s(h,f)
\hbox{ and }
N(f,H)
:=-3+\frac{2}{f}+12S'(H,f).
\end{equation}
Then, for $p\geq 3$ a prime 
and $H$ a subgroup of odd order of the multiplicative group $({\mathbb Z}/p{\mathbb Z})^*$, we have
\begin{equation}\label{M(p,H)}
M(p,H)
=\frac{\pi^2}{6}
\left (1+\frac{N(p,H)}{p}\right )
=\frac{\pi^2}{6}\left (\left (1-\frac{1}{p}\right )\left (1-\frac{2}{p}\right )+\frac{12S'(H,p)}{p}\right ).
\end{equation}
\end{proposition}
}}

\begin{Remark}
In particular, $N(f,\{1\}) =-3+2/f$ and (\ref{M(p,H)}) implies (\ref{Mp1}). 
Notice also that $N(p,H)\in {\mathbb Z}$ for $H\neq\{1\}$, by \cite[Theorem 6]{LouBKMS56}.
Moreover, by \cite[Theorem $1.1$]{LMQJM}, the asymptotic formula 
$M(p,H) 
=\frac{\pi^2}{6}+o(1)$ holds
as $p$ tends to infinity and $H$ runs over the subgroup of $({\mathbb Z}/p{\mathbb Z})^*$ 
of odd order $d\leq \frac{\log p}{\log \log p}$. 
Hence we have $N(p,H)=o(p)$ under this restriction.
\end{Remark}

\subsection{A formula for $M_{d_0}(p,H)$}\label{sectionformulad0}
We will now derive a third consequence of (\ref{formulaM(f,H)}): 
a formula for the mean square value $M_{d_0}(f,H)$ defined in \eqref{Md0(f,H)} when $f$ is prime.

\noindent\frame{\vbox{
\begin{theorem}\label{Thgeneralexplicit}
Let $d_0>1$ be a square-free integer. 
Let $f>2$ be coprime with $d_0$.
Let $H$ be a subgroup of the multiplicative group $({\mathbb Z}/f{\mathbb Z})^*$. 
Whenever $\delta$ divides $d_0$, 
let $s_\delta:({\mathbb Z}/\delta f{\mathbb Z})^*\longrightarrow ({\mathbb Z}/f{\mathbb Z})^*$ 
be the canonical surjective morphism 
and set $H_{\delta}=s_\delta^{-1}(H)$ and $H_{\delta}' =s_\delta^{-1}(H\setminus\{1\})$. 
Define the rational number
\begin{equation}\label{defNd0fH}
N_{d_0}(f,H)
=-f
+\frac{12\mu(d_0)}{\prod_{q\mid d_0}(q^2-1)}
\sum_{\delta\mid d_0}\delta\mu (\delta)\sum_{h\in H_{d_0}}s(h,\delta f).
\end{equation}
Then, for $p\geq 3$ a prime which does not divide $d_0$ 
and $H$ a subgroup of odd order of the multiplicative group $({\mathbb Z}/p{\mathbb Z})^*$, we have 
\begin{equation}\label{formulaMd0pH}
M_{d_0}(p,H)
=\frac{2\pi^2\mu(d_0)\phi(d_0)}{d_0^2p}
\sum_{\delta\mid d_0}\frac{\delta\mu(\delta)}{\phi(\delta)}
S(H_\delta,\delta p)\end{equation}
where 
\begin{equation*} 
S(H_\delta,\delta f) 
=\sum_{h\in H_\delta}s(h,\delta f),
\end{equation*}
and
\begin{equation}\label{Md0pH}
M_{d_0}(p,H) 
=\kappa_{d_0}
\times\left (
1
+\frac{N_{d_0}(p,H)}{p}
\right ), 
\hbox{ where }
\kappa_{d_0}
:=\frac{\pi^2}{6}
\prod_{q\mid d_0}\left (1-\frac{1}{q^2}\right ).
\end{equation}
Moreover, 
\begin{align}
N_{d_0}(f,H)
&=-f
+\frac{12\mu(d_0)}{\prod_{q\mid d_0}(q+1)}
\sum_{\delta\mid d_0}\frac{\delta\mu(\delta)}{\phi(\delta)}S(H_\delta,\delta f)\label{defNd0fHbis}\\
&=N_{d_0}(f,\{1\})
+\frac{12\mu(d_0)}{\prod_{q\mid d_0}(q+1)}
\sum_{\delta\mid d_0}\frac{\delta\mu(\delta)}{\phi(\delta)}
S'(H_{\delta},\delta f)\label{defNd0fHter}
\end{align} 
where
\begin{equation*}
S'(H_{\delta},\delta f)
:=\sum_{h\in H_{\delta}'}s(h,\delta f). 
\end{equation*}
\end{theorem}
}}

\begin{proof}
Using (\ref{MM}) and by making the change of variables $\delta\mapsto d_0f/\delta$ in (\ref{formulaM(f,H)}), we obtain: 
\begin{equation}\label{Md0fH}
M_{d_0}(f,H)
=M(d_0f,H_{d_0})
={2\pi^2\over d_0^2f^2}\sum_{\delta\mid d_0f}\delta\mu (d_0f/\delta)\sum_{h\in H_{d_0}}s(h,\delta).
\end{equation}
Since
$\{\delta;\ \delta\mid d_0p\}$ 
is the disjoint union of
$\{\delta;\ \delta\mid d_0\}$
and $\{\delta p;\ \delta\mid d_0\}$, 
by (\ref{Md0fH}) we obtain:
$$M_{d_0}(p,H)
=-{2\pi^2\mu(d_0)\over d_0^2p^2}\sum_{\delta\mid d_0}\delta\mu (\delta)\sum_{h\in H_{d_0}}s(h,\delta)
+{2\pi^2\mu(d_0)\over d_0^2p}\sum_{\delta\mid d_0}\delta\mu (\delta)\sum_{h\in H_{d_0}}s(h,\delta p).$$
Now, 
$S:=\sum_{h\in H_{d_0}}s(h,\delta)=0$ whenever $\delta\mid d_0$, 
which gives 
\begin{equation}\label{stepMersenne1}
M_{d_0}(p,H)
={2\pi^2\mu(d_0)\over d_0^2p}\sum_{\delta\mid d_0}\delta\mu (\delta)\sum_{h\in H_{d_0}}s(h,\delta p)
\end{equation} 
and implies (\ref{Md0pH}).
Indeed, let $\sigma:({\mathbb Z}/d_0f{\mathbb Z})^*\longrightarrow ({\mathbb Z}/\delta{\mathbb Z})^*$ 
be the canonical surjective morphism. 
Its restriction $\tau$ to the subgroup $H_{d_0}$ is surjective, 
by the Chinese reminder theorem. 
Hence,
$S
=(H_{d_0}:\ker\tau)
\times S'$, 
where 
$S':=\sum_{c\in ({\mathbb Z}/\delta{\mathbb Z})^*}s(c,\delta)
=\sum_{c\in ({\mathbb Z}/\delta{\mathbb Z})^*}s(-c,\delta)
=-S'$ 
yields $S'=0$.
In the same way, whenever $\delta\mid d_0$,
the kernel of the canonical surjective morphism 
$s:({\mathbb Z}/d_0f{\mathbb Z})^*\longrightarrow ({\mathbb Z}/\delta f{\mathbb Z})^*$ 
being a subgroup of order $\phi(d_0f)/\phi(\delta f) =\phi(d_0)/\phi(\delta)$, 
we have
\begin{equation}\label{stepMersenne2}
\sum_{h\in H_{d_0}}s(h,\delta f)
=\frac{\phi(d_0)}{\phi(\delta)}\sum_{h\in H_{\delta}}s(h,\delta f)
\end{equation} 
and (\ref{formulaMd0pH}) follows from (\ref{stepMersenne1}) and (\ref{stepMersenne2}). 

Then, (\ref{Md0pH}) is a direct consequence of (\ref{formulaMd0pH}) and (\ref{defNd0fH}). 
Finally (\ref{defNd0fHter}) is an immediate consequence of (\ref{defNd0fH}) and (\ref{stepMersenne2}).
\end{proof}

\subsubsection{A new proof of Theorem \ref{asympd0}}
We split the sum in \eqref{stepMersenne1} into two cases depending whether $h=1$ or not. 
By \eqref{s1d} we have $s(1,\delta p)=\frac{p\delta}{12}+O(1)$ giving a contribution to the sum of order
$$ \frac{\pi^2 \mu(d_0)}{6d_0^2} \sum_{\delta \mid d_0}\delta^2\mu(\delta) + O(1/p) 
= \frac{\pi^2}{6} \prod_{q \mid d_0}\left(1-\frac{1}{q^2}\right)+O(1/p).$$ 
When $h\neq 1$ and $h\in H_{d_0}$, it is clear that the order of $h$ modulo $p$ is between $3$ and $d$. 
Hence it follows from Theorem \ref{indivbound} (see the Remark after) that 
$s(h,\delta p)=O((\log p)^2 p^{1-\frac{1}{\phi(d)}})$. 
The integer $d_0$ being fixed, we can sum up these error terms and the proof is finished.

\subsection{An explicit way to compute $N_{d_0}(f,\{1\})$}
\begin{lemma}\label{Sd0fAd0f}
Let $d_0>1$ be a square-free integer. 
Let $f>2$ be coprime with $d_0$.
Recall that $H_{d_0}(f)
=\{h\in ({\mathbb Z}/d_0f{\mathbb Z})^*,\ h\equiv 1\pmod f\}$ and set
$$U(d_0,f)
:=\sum_{1\neq h\in H_{d_0}(f)}
\sum_{n=1\atop\gcd(d_0,n)=1}^{d_0f -1}
\left (1+\cot\left ({\pi n\over d_0f}\right )\cot\left ({\pi nh\over d_0f}\right )\right )$$
and
\begin{equation}\label{Ad0f}
A(d_0,f)
=\sum_{a\in ({\mathbb Z}/d_0{\mathbb Z})^*}
\sum_{b\in ({\mathbb Z}/d_0{\mathbb Z})^*\atop b\neq a}
\cot\left (\frac{\pi (b-a)}{d_0}\right )
\left (
\cot\left (\frac{\pi fa}{d_0}\right )
-\cot\left (\frac{\pi fb}{d_0}\right )
\right ),
\end{equation}
a rational number depending only on $f$ modulo $d_0$. 
Then 
$U(d_0,f)
=fA(d_0,f).$ 
\end{lemma}

\begin{proof}
As in \cite[Lemma ]{LouPMDebr78}, 
set
$$T(d_0,f)
:=\sum_{1\neq h\in H_{d_0}(f)}
\sum_{n=1\atop\gcd(d_0f,n)=1}^{d_0f -1}
F\left (\frac{n}{d_0f},\frac{nh}{d_0f}\right ),$$
where $F(x,y)=1+\cot (\pi x)\cot (\pi y)$. 
On the one hand, 
since $\gcd(d_0f,n)=1$ if and only if $\gcd(d_0,n)=\gcd(f,n)=1$ 
and $\displaystyle{\sum_{d\mid f\atop d\mid n}\mu(d)}$ is equal to $1$ if $\gcd(f,n)=1$ and is equal to $0$ otherwise, 
we have
$$T(d_0,f)
=\sum_{d\mid f}\mu(d)
\sum_{1\neq h\in H_{d_0}(f)}
\sum_{n=1\atop\gcd(d_0,n)=1}^{d_0(f/d) -1}F\left (\frac{n}{d_0(f/d)},\frac{nh}{d_0(f/d)}\right ).$$
On the other hand, 
the canonical morphism 
$\sigma :H_{d_0}(f)
\rightarrow H_{d_0}(f/d)$ 
is surjective 
and both groups have order $\phi(d_0f)/\phi(f) =\phi(d_0(f/d))/\phi (f/d) =\phi(d_0)$. 
Hence $\sigma$ is bijective and 
$$T(d_0,f)
=\sum_{d\mid f}\mu(d)U(d_0,f/d).$$
Using \cite[Lemma 6]{LouPMDebr78} and M\"{o}bius' inversion formula, 
we finally do obtain 
\begin{align*}
U(d_0,f)
=\sum_{d\mid f}T(d_0,d)
&=\sum_{d\mid f}d\sum_{\delta\mid d}\frac{\mu(\delta)}{\delta}A(d_0,d/\delta)\\
&=\sum_{\delta'\mid f}
\delta'\left (\sum_{\delta\mid f/\delta'}\mu(\delta)\right )A(d_0,\delta')
=fA(d_0,f),
\end{align*}
where we set $\delta'=d/\delta$.
\end{proof}

\begin{proposition}\label{H=1}
Let $d_0>1$ be a square-free integer. 
Set 
$B=\prod_{q\mid d_0}(q^2-1)$. 
For $f>2$ and $\gcd (d_0,f)=1$ we have 
$$N_{d_0}(f,\{1\})
=\frac{3}{B}
\left (A(d_0,f)-\phi(d_0)^2\right ).$$
Consequently, $N_{d_0}(f,\{1\})$ is a rational number depending only on $f$ modulo $d_0$. 
\end{proposition}

\begin{proof}
Set $H=H_{d_0}(f):=\{h\in ({\mathbb Z}/d_0f{\mathbb Z})^*,\ h\equiv 1\pmod f\}$.
By (\ref{defNd0fH}) we have 
$$N_{d_0}(f,\{1\})
=-f+\frac{12\mu(d_0)}{B}\sum_{\delta\mid d_0}\delta\mu (\delta)
\sum_{h\in H}s(h,\delta f).$$
Using (\ref{s1d}) 
to evaluate the contribution of $h=1$ in this expression 
and $\sum_{\delta\mid d_0}\mu(\delta) =0$, 
we get 
$$N_{d_0}(f,\{1\})
=-\frac{3\phi(d_0)}{B}
+\frac{12\mu(d_0)}{B}\sum_{\delta\mid d_0}\delta\mu (\delta)
\sum_{1\neq h\in H}s(h,\delta f)$$
and 
$$N_{d_0}(f,\{1\})
=-\frac{3\phi(d_0)^2}{B}
+\frac{3\mu(d_0)}{Bf}
\sum_{1\neq h\in H}
\sum_{\delta\mid d_0}\mu (\delta)
\sum_{n=1}^{\delta f -1}\left (1+\cot\left ({\pi n\over\delta f}\right )\cot\left ({\pi nh\over\delta f}\right )\right ),$$
by (\ref{defscd}) 
and by noticing that $\# H =\phi(d_0)$. 
Therefore, 
\begin{equation}\label{S}
N_{d_0}(f,\{1\})
=-\frac{3\phi(d_0)^2}{B}
+\frac{3}{Bf}S(d_0,f)
\end{equation}
(make the change of variable $\delta\mapsto d_0/\delta$). 
Lemma \ref{Sd0fAd0f} gives the desired result.
\end{proof}

\begin{Remark}
As a consequence we obtain
$M_{d_0}(p,\{1\})
=\frac{\pi^2}{6}\prod_{q\mid d_0}\left (1-\frac{1}{q^2}\right )
+O(p^{-1})$, 
using (\ref{Md0pH}) and the fact that $N_{d_0}(p,\{1\})$ 
depends only on $p$ modulo $d_0$. 
This gives in this extreme situation another proof of Theorem \ref{asympd0} with a better error term.
Moreover, in that situation we have $K ={\mathbb Q}(\zeta_p)$ and in (\ref{boundhKminusd0}) 
the term $\Pi_{d_0}(p,\{1\})$ is bounded from below by a constant independent of $p$, 
by Corollary \ref{PiKminus}.
\end{Remark}

\section{The case where $f=a^{d-1}+\cdots+a^2+a+1$}\label{genMersenne}
In this specific case we are able to obtain explicit formulas for $M_{d_0}(f,H)$ 
when the subgroup $H$ is defined in terms of the parameter $a$ defining the modulus. 
For a general subgroup $H$, it seems unrealistic to be more explicit than the formula involving Dedekind sums 
given in Theorem \ref{Thgeneralexplicit}. 
It might be interesting to explore formulas involving continued fraction expansions 
in view of their link to Dedekind sums \cite{Hic}.

\subsection{Explicit formulas for $d_0=1,2$}
\begin{lemma}\label{(a^d-1)/(a-1)}
Let $f>1$ be a rational integer of the form $f=(a^d-1)/(a-1)$ for some $a\neq -1,0,1$ 
and some odd integer $d\geq 3$. Hence $f$ is odd.
Set $H=\{a^k;\ 0\leq k\leq d-1\}$, 
a subgroup of order $d$ of the multiplicative group $({\mathbb Z}/f{\mathbb Z})^*$.
Then, 
$$S(H,f)
=\frac{a+1}{a-1}\times\frac{f-(d-1)a-1}{12}$$
and 
$$S(H_2,2f)
=\begin{cases}
\frac{a+1}{a-1}\times\frac{4f-(d-1)a-3d-1}{24}&\hbox{if $a$ is odd}\\
\frac{2a-1}{a-1}\times\frac{f-(d-1)a-1}{12}&\hbox{if $a$ is even.}
\end{cases}$$
\end{lemma} 

\begin{proof}
We have 
$S(H,f)
=\sum_{k=0}^{d-1}s(a^k,f)$. 
Moreover, 
$S(H_2,2f)
=\sum_{k=0}^{d-1}s(a^k,2f)$ if $a$ is odd
and 
$S(H_2,2f)
=s(1,2f)+\sum_{k=1}^{d-1}s(a^k+f,2f)$ if $a$ is even.
Now, we claim that for $0\leq k\leq d-1$ we have
$$s(a^k,f)
=\frac{a^k}{12f}+\frac{(f^2+1)a^{-k}}{12f}
+\frac{a^k+a^{-k}(a^2-2a+2)}{12(a-1)}
-\frac{a(a+1)}{12(a-1)} 
\hbox{ whatever the parity of $a$,}$$ 
$$s(a^k,2f)
=\frac{a^k}{24f}+\frac{(4f^2+1)a^{-k}}{24f}
+\frac{4a^k+a^{-k}(a^2-2a+5)}{24(a-1)}
-\frac{(a+1)(a+3)}{24(a-1)} 
\hbox{ if $a$ is odd,}$$
and that for $1\leq k\leq d-1$ we have
$$s(a^k+f,2f)
=\frac{a^k}{24f}+\frac{(f^2+1)a^{-k}}{24f}
+\frac{a^k+a^{-k}(a^2-2a+2)}{24(a-1)}
-\frac{a(2a-1)}{12(a-1)} 
\hbox{ if $a$ is even.}$$ 
Noticing that 
$\sum_{k=1}^{d-1}a^k =f-1$ 
and 
$\sum_{k=1}^{d-1}a^{-k} =\frac{f-1}{(a-1)f+1}$, 
we then get the assertions on $S(H,f)$ and $S(H_2,2f)$. 
Now, let us for example prove the third claim.
Hence, assume that $a$ is even and that $1\leq k\leq d-1$. 
Then $f_k :=(a^k-1)/(a-1)$ is odd, 
${\rm sign}(f_k)={\rm sign}(a)^k$ and $a^k+f>0$.\\
First, since $2f\equiv -2a^{k}\pmod {a^k+f}$, 
using (\ref{scddc}) we have 
$$s(a^k+f,2f) 
=\frac{a^k+f}{24f}
+\frac{f}{6(a^k+f)}
-\frac{1}{4}
+\frac{1}{24(a^k+f)f}
+s(2a^{k},a^k+f).$$
Second, noticing that $a^k+f\equiv f_k\pmod {2a^{k}}$ and using (\ref{scddc}) we have 
$$s(2a^{k},a^k+f)
=\frac{a^{k}}{6(a^k+f)}
+\frac{a^k+f}{24a^{k}}
-\frac{{\rm sign}(a)^k}{4}
+\frac{1}{24a^{k}(a^k+f)}
-s(f_k,2a^{k}).$$
Finally, noticing that $2a^{k}\equiv 2\pmod{f_k}$ and using (\ref{scddc}) and (\ref{s1d}) we have 
\begin{eqnarray*}
s(f_k,2a^{k})
&=&\frac{f_k}{24a^{k}}
+\frac{a^{k}}{6f_k}
-\frac{{\rm sign}(a)^k}{4}
+\frac{1}{24f_ka^{k}}
-s(2,f_k)\\
&=&\frac{f_k}{24a^{k}}
+\frac{a^{k}}{6f_k}
-\frac{{\rm sign}(a)^k}{4}
+\frac{1}{24f_ka^{k}}
-\frac{f_k^2-6f_k+5}{24f_k}.
\end{eqnarray*}
After some simplifications, 
we obtain the desired formula for $s(a^k+f,2f)$. 

Notice that for $d=3$ we obtain $S(H,f) =\frac{f-1}{12}$, 
in accordance with (\ref{SfabH}). 
\end{proof}

Using \eqref{formulaMd0pH} and Lemma \ref{(a^d-1)/(a-1)} we readily obtain:

\noindent\frame{\vbox{
\begin{theorem}\label{TheoremMersenne}
Let $d\geq 3$ be a prime integer. 
Let $p\equiv 1\pmod{2d}$ be a prime integer of the form $p=(a^d-1)/(a-1)$ for some $a\neq -1,0,1$. 
Let $K$ be the imaginary subfield of degree $(p-1)/d$ of the cyclotomic field ${\mathbb Q}(\zeta_p)$.
Set $H=\{a^k;\ 0\leq k\leq d-1\}$, 
a subgroup of order $d$ of the multiplicative group $({\mathbb Z}/p{\mathbb Z})^*$.
We have the mean square value formulas
\begin{equation}\label{Mpd}
M(p,H)
=\frac{\pi^2}{6}\times\frac{a+1}{a-1}\times\left (1-\frac{(d-1)a+1}{p}\right ).
\end{equation}
and
\begin{equation}\label{M2pd}
M_2(p,H)
=\frac{\pi^2}{8}\times
\begin{cases}
\frac{a+1}{a-1}\times\left (1-\frac{d}{p}\right )&\hbox{ if $a$ is odd,}\\
1-\frac{(d-1)a+1}{p}&\hbox{ if $a$ is even.}\\
\end{cases}
\end{equation}
Consequently, for a given $d$, 
as $p\rightarrow\infty$ we have 
\begin{equation*}
M(p,H)
=\frac{\pi^2}{6}+o(1)
\hbox{ and }
M_2(p,H)
=\frac{\pi^2}{8}+o(1).
\end{equation*}
On the other hand, for a given $a$, as $p\rightarrow\infty$ we have 
\begin{equation*}
M(p,H)
=\frac{\pi^2}{6}\times\frac{a+1}{a-1}+o(1)
\hbox{ and }
M_2(p,H)
=\begin{cases}
\frac{\pi^2}{8}\times\frac{a+1}{a-1}+o(1)
&\hbox{if $a$ is odd,}\\
\frac{\pi^2}{8}+o(1)
&\hbox{if $a$ is even.}
\end{cases}
\end{equation*}
\end{theorem}
}}
\begin{Remark}
Assertion (\ref{Mpd}) was initially proved\footnote{Note the misprint in the exponent in \cite[Theorem 5]{LouBPASM64}} 
in \cite[Theorem 5]{LouBPASM64} for $d=5$
and then generalized in \cite[Proposition 3.1]{LMQJM} to any $d\geq 3$.
However, (\ref{Mpd}) is much simpler than \cite[(22)]{LMQJM}. 
Notice that if $p$ runs over the prime of the form $p=(a^d-1)/(a-1)$
with $a\neq 0,2$ even then 
$M_2(p,H)
=\frac{6}{8}\times\frac{a-1}{a+1}\times M(p,H)$ 
and the asymptotic (\ref{asymptoticMd0(p,H)M(p,H)}) is not satisfied.
\end{Remark}

\subsection{The case where $p$ is a Mersenne prime and $d_0=1,3,15$}\label{Mersennesection}
In the setting of Theorem \ref{Mersenne}, we have $2\in H$. 
Hence, by Remark \ref{restrictiond0} we assume that $d_0$ is odd.

\noindent\frame{\vbox{
\begin{theorem}\label{Mersenne}
Let $p=2^d-1>3$ be a Mersenne prime. 
Hence, $d$ is odd and $H=\{2^k;\ 0\leq k\leq d-1\}$ 
is a subgroup of odd order $d$ of the multiplicative group $({\mathbb Z}/p{\mathbb Z})^*$.
Let $K$ be the imaginary subfield of degree $m =(p-1)/d$ of ${\mathbb Q}(\zeta_p)$. 
Then 
$$M(p,H)
=\frac{\pi^2}{2}\left (1-\frac{2d-1}{p}\right )
\leq\frac{\pi^2}{2}
\text{ and } 
h_K^-\leq 2\left (\frac{p}{8}\right )^{m/4},$$
$$M_3(p,H)
=\frac{4\pi^2}{9}\left (1-\frac{d}{p}\right )
\leq\frac{4\pi^2}{9}
\text{ and } 
h_K^-\leq 2\left (\frac{p}{9}\right )^{m/4}$$
and 
$$M_{15}(p,H)
=\frac{32\pi^2}{75}\left (1-\frac{c_d}{48p}\right )
\leq\frac{32\pi^2}{75},
\hbox{ where }
c_d
=\begin{cases}
47d+1&\hbox{if $d\equiv 1\pmod 4$,}\\
17d-3&\hbox{if $d\equiv 3\pmod 4$.}
\end{cases}$$
In particular, for $d\equiv 3\pmod 4$ we have
$h_K^-\leq 2\left (\frac{8p}{75}\right )^{m/4}.$
\end{theorem}
}}

\begin{proof}
By (\ref{formulaMd0pH}) we have
\begin{equation}\label{Md0Nd0'(f,H)}
M_{d_0}(p,H)
=\frac{\pi^2}{2}\left\{\prod_{q\mid d_0}\left (1-\frac{1}{q^2}\right )\right\}
\left (1+\frac{N_{d_0}'(p,H)}{p}\right ),
\end{equation}
where for $H$ a subgroup of odd order of the multiplicative group $({\mathbb Z}/f{\mathbb Z})^*$ we set
\begin{equation}\label{Nd0'(f,H)}
N_{d_0}'(f,H)
:=-f
+\frac{4\mu(d_0)}{\prod_{q\mid d_0}(q+1)}
\sum_{\delta\mid d_0}\frac{\delta\mu(\delta)}{\phi(\delta)}S(H_\delta,\delta f) .
\end{equation}
The formulas for $M(p,H),M_3(p,H)$ and $M_{15}(p,H)$ follow from (\ref{Md0Nd0'(f,H)}) and Lemma \ref{Mersennesq=1,3,5,15} below. 
The upper bounds on $h_K^-$ follow from (\ref{boundhKminusd0}) and Lemma \ref{formulaPi} 
according to which $\Pi_q(p,H)\geq 1$ if $q$ is of even order in the quotient group $G/H$, 
where $G=({\mathbb Z}/p{\mathbb Z})^*$,
hence if $q$ is of even order in the group $G$. 
Now, since $p\equiv 3\pmod 4$ the group $G$ is of order $p-1=2N$ with $N$ odd 
and $q$ is of even order in $G$ if and only $q^N=-1$ in $G$, 
i.e. if and only if the Legendre symbol $\left (\frac{q}{p}\right )$ is equal to $-1$. 
Now, since $p=2^d-1\equiv-1\equiv 3 \pmod 4$ for $d\geq 3$, the law of quadratic reciprocity gives 
$\left (\frac{3}{p}\right ) =-\left (\frac{p}{3}\right ) =-\left (\frac{1}{3}\right ) =-1$, 
as $p\equiv (-1)^d-1\equiv -2\equiv 1\pmod 3$. Hence, $\Pi_3(p,H)\geq 1$. 
In the same way, if $d\equiv 3\pmod 4$ then 
$p=2^d-1 =2\cdot 4^\frac{d-1}{2}-1\equiv 2\cdot (-1)^\frac{d-1}{2}-1\equiv -3\equiv 2\pmod 5$ 
and $\left (\frac{5}{p}\right ) =\left (\frac{p}{5}\right ) =\left (\frac{2}{5}\right ) =-1$ 
and $\Pi_5(p,H)\geq 1$.
\end{proof}

\begin{lemma}\label{Mersennesq=1,3,5,15}
Set $f=2^d-1$ and $\varepsilon_d =(-1)^{(d-1)/2}$ with $d\geq 2$ odd. 
Hence $\gcd(f,15)=1$.
Set 
$H 
=\{2^k;\ 0\leq k\leq d-1\}$, 
a subgroup of order $d$ of the multiplicative group $({\mathbb Z}/f{\mathbb Z})^*$. 
Then,
\begin{equation}\label{S(H,fMersenne)}
S(H,f) 
=\frac{f-2d+1}{4}
\hbox{ and }
N'(f,H)
=-2d+1,
\end{equation} 
\begin{equation}\label{S(H3,3fMersenne)}
S(H_3,3f) 
=\frac{5f-6d+1}{6}
\hbox{ and }
N_3'(f,H) =-d,
\end{equation}
\begin{equation}\label{S(H5,5fMersenne)}
S(H_5,5f) 
=\frac{7f-10d+2+\varepsilon_d}{5} 
\hbox{ and }
N_5'(f,H) 
=-\frac{4}{3}d+\frac{1+\varepsilon_d}{6},
\end{equation}
\begin{equation}\label{S(H15,15fMersenne)}
S(H_{15},15f) 
=\frac{14f-\left (12+3\varepsilon_d\right )d+1}{3}
\hbox{ and }
N_{15}'(f,H)
=-\frac{32+15\varepsilon_d}{48}d
+\frac{1-2\varepsilon_d}{48}.
\end{equation}
\end{lemma}

\begin{proof}
The first assertion is the special case $a=2$ of Lemma \ref{(a^d-1)/(a-1)}. 
Let us now deal with the second assertion.
Here 
$H_3
=\{2^k;\ 0\leq k\leq d-1\}
\cup\{2^k+(-1)^kf;\ 0\leq k\leq d-1\}$.
We assume that $0\leq k\leq d-1$. 
Hence, ${\rm sign}(2^k+(-1)^kf)=(-1)^k$.

\noindent {\bf 1.} 
Noticing that $3f\equiv -3\pmod {2^k}$, 
by (\ref{scddc}) we obtain 
$$s(2^k,3f)
=\frac{4^k+9f^2-9\cdot 2^k\cdot f+1}{36\cdot 2^k\cdot f}+s(3,2^k).$$
Noticing that $2^k\equiv (-1)^k\pmod 3$, 
by (\ref{scddc}) and (\ref{s1d}) we obtain 
$$s(3,2^k)
=\frac{9+4^k-9\cdot 2^k+1}{36\cdot 2^k}-(-1)^ks(1,3)
=\frac{9+4^k-9\cdot 2^k+1}{36\cdot 2^k}-\frac{(-1)^k}{18}.$$
Hence 
$$s(2^k,3f)
=\frac{f+1}{36f}2^k
+\frac{(f+1)(9f+1)}{36f}2^{-k}
-\frac{1}{2}
-\frac{(-1)^k}{18}.$$ 

\noindent {\bf 2.} 
Noticing that $3f\equiv -3\cdot (-1)^k2^k\pmod {2^k+(-1)^kf}$, 
by (\ref{scddc}) we obtain 
\begin{align*}
s(2^k+(-1)^kf,3f)
=&\frac{2^k+(-1)^kf}{36f}
+\frac{f}{4(2^k+(-1)^kf)}
-\frac{(-1)^k}{4}
+\frac{1}{36(2^k+(-1)^kf)f}\\
&+(-1)^ks(3\cdot 2^k,2^k+(-1)^kf)
\end{align*}
and noticing that $2^k+(-1)^kf\equiv (-1)^{k-1}\pmod {3\cdot 2^k}$, 
by (\ref{scddc}) we obtain 
\begin{align*}
s(3\cdot 2^k,2^k+(-1)^kf)
=&\frac{3\cdot 2^k}{12(2^k+(-1)^kf)}
+\frac{2^k+(-1)^kf}{36\cdot 2^k}
-\frac{(-1)^k}{4} \\
&+\frac{1}{36\cdot 2^k\cdot (2^k+(-1)^kf)}+(-1)^ks(1,3\cdot 2^k).
\end{align*}
Using (\ref{s1d}) we finally obtain
$$s(2^k+(-1)^kf,3f)
=\frac{9f+1}{36f}2^k
+\frac{(f+1)^2}{36f}2^{-k}
-\frac{1}{2}
+\frac{(-1)^k}{18}.$$

\noindent {\bf 3.}
Using $\sum_{k=0}^{d-1}2^k =f$, $\sum_{k=0}^{d-1}2^{-k}=\frac{2f}{f+1}$ and $\sum_{k=0}^{d-1}(-1)^k=1$,
we obtain 
$$\sum_{k=0}^{d-1}s(2^k,3f)
=\frac{19f-18d+1}{36}
\hbox{ and }
\sum_{k=0}^{d-1}s(2^k+(-1)^kf,3f)
=\frac{11f-18d+5}{36}.$$
Hence, we do obtain
$$S(H_3,3f) 
=\sum_{h\in H_3}s(h,3f)
=\frac{19f-18d+1}{36}
+\frac{11f-18d+5}{36}
=\frac{5f-6d+1}{6}$$
and $N_3'(f,H) =-d$, by (\ref{Nd0'(f,H)}).

Let us finally deal with the third and fourth assertions. 
The proof involves tedious and repetitive computations. 
For this reason we will restrict ourselves to a specific case.
Let us for example give some details for the proof of (\ref{S(H15,15fMersenne)}) 
in the case that $d\equiv 1\pmod 4$. 
We have $f=2^d-1\equiv 1\pmod {30}$ and 
$H_{15}=\cup_{l=0}^{14}E_l$, 
where 
$E_l
:= \{2^k+lf;\ 0\leq k\leq d-1,\ \gcd(2^k+l,15)=1\}$ for $0\leq l\leq 14$. 
We have to compute the sums $s_l:=\sum_{n\in E_l}s(n,15f)$. 
Let us for example give some details in the case that $l=1$. 
We have $\gcd (2^k+1,15)=1$ if and only if $k\equiv 0\pmod 4$. 
Hence $s_1=\sum_{k=0}^{(d-1)/4}s(16^{k}+f,15f)$. 
Using (\ref{scddc}) and (\ref{s1d}) we obtain
$$s(16^k+f,15f)
=\frac{9f+1}{180f}16^k
+\frac{14}{45}
+\frac{(f+1)^2}{180f}16^{-k}.$$
Finally, using
$\sum_{k=0}^{(d-1)/4}16^k
=\frac{8f+7}{15}$ 
and 
$\sum_{k=0}^{(d-1)/4}16^{-k}
=\frac{2(8f+7)}{15(f+1)}$ 
we obtain
$$s_1
=\sum_{k=0}^{(d-1)/4}s(16^k+f,15f)
=\frac{88f^2+(210d+731)f+21}{2700f}.$$
Finally, using 
(\ref{Nd0'(f,H)}), 
(\ref {S(H,fMersenne)}), 
(\ref{S(H3,3fMersenne)})
and (\ref{S(H5,5fMersenne)}) 
we get (\ref{S(H15,15fMersenne)}).
\end{proof}

We conclude this Section with the following result for $d_0 =3\cdot 5\cdot 7 =105$, 
whose long proof we omit\footnote{The formulas can be and have been checked on numerous examples 
using a computer algebra system. 
Indeed, by (\ref{scddc}) and (\ref{s1d}) any Dedekind sum $s(c,d)\in {\mathbb Q}$ with $c,d\geq 1$ 
can be easily computed 
by successive euclidean divisions of $c$ by $d$ and exchanges of $c$ and $d$, 
until we reach $c=1$. }:

\begin{lemma}\label{Mersennesq=105}
Set $f =2^d-1$ with $d>1$ odd. 
Assume $\gcd (f,105)=1$, i.e. that $d\equiv 1, 5, 7, 11\pmod{12}$.
Set 
$H 
=\{2^k;\ 0\leq k\leq d-1\}$, 
a subgroup of order $d$ of the multiplicative group $({\mathbb Z}/f{\mathbb Z})^*$. 
Then 
\begin{equation*}
N_{105}'(f,H)
=-\frac{1}{576}
\times\begin{cases}
437d+139&\hbox{if $d\equiv 1\pmod {12}$},\\
535d-644&\hbox{if $d\equiv 5\pmod{12}$},\\
97d-324&\hbox{if $d\equiv 7\pmod{12}$},\\
195d+13&\hbox{if $d\equiv 11\pmod {12}$.}
\end{cases}
\end{equation*}
\end{lemma}

Lemmas \ref {Mersennesq=1,3,5,15}-\ref{Mersennesq=105} 
show that the following Conjecture holds true for $d_0\in\{1,3,5,15,105\}$:

\noindent\frame{\vbox{
\begin{Conjecture}\label{conjectureMersenne}
Let $d_0\geq 1$ be odd and square-free. 
Let $N$ be the order of $2$ in the multiplicative group $({\mathbb Z}/d_0{\mathbb Z})^*$. 
Set $f =2^d-1$ with $d>1$ odd 
and
$H 
=\{2^k;\ 0\leq k\leq d-1\}$, 
a subgroup of order $d$ of the multiplicative group $({\mathbb Z}/f{\mathbb Z})^*$. 
Assume $\gcd (f,d_0)=1$. 
Then $N_{d_0}'(f,H)
=A_1(d)d+A_0(d)$, 
where $A_1(d)$ and $A_0(d)$ are rational numbers which depend only on $d$ modulo $N$, 
i.e. only on $f$ modulo $d_0$. Hence for a prime $p\geq 3$ we expect
$$M_{d_0}(p,H)
=\frac{\pi^2}{2}\left\{\prod_{q \mid d_0}
\left(1-\frac{1}{q^2}\right)\right\}\left(1+\frac{A_1(d)d}{p}+\frac{A_0(d)}{p}\right),$$ 
confirming again that the restriction on $d$ in Theorem \ref{asympd0} should be sharp.
\end{Conjecture}
}}

There is apparently no theoretical obstruction preventing us to prove Conjecture \ref{conjectureMersenne}. 
Indeed, for a fixed $d_0$, the formulas for $A_0(d)$ and $A_1(d)$ could be guessed 
using numerous examples on a computer algebra system. 
However for large $d_0$'s the set of cases to consider grows linearly 
and a more unified approach seems to be required to give a complete proof.

\section{The case of subgroups of order $d=3$}\label{Sectiond=3}
\subsection{Formulas for $d_0=1,2,6$}
Let $p\equiv 1\pmod 6$ be a prime integer.
Let $K$ be the imaginary subfield of degree $m=(p-1)/3$ of the cyclotomic field ${\mathbb Q}(\zeta_p)$.
Since $p$ splits completely in the quadratic field ${\mathbb Q}(\sqrt{-3})$ 
of class number one, 
there exists an algebraic integer $\alpha =a+b\frac{1+\sqrt{-3}}{2}$ 
with $a,b\in {\mathbb Z}$ such that $p =N_{{\mathbb Q}(\sqrt{-3})/{\mathbb Q}}(\alpha) =a^2+ab+b^2$. 
Then, 
$H=\{1,a/b,b/a\}$, is the unique subgroup of order $3$ 
of the cyclic multiplicative group $({\mathbb Z}/p{\mathbb Z})^*$. 
So we consider the integers $f>3$ of the form 
$f=a^2+ab+b^2$, with $a,b\in {\mathbb Z}\setminus\{0\}$ and 
$\gcd(a,b)=1$,
which implies $\gcd(a,f) =\gcd (b,f)=1$ and the oddness of $f$. 
We have the following explicit formula.

\begin{lemma}\label{sabf}
Let $f>3$ be of the form 
$f=a^2+ab+b^2$, with $a,b\in {\mathbb Z}$ and $\gcd(a,b)=1$. 
Set $H=\left\{1,a/b,b/a\right\}$, a subgroup of order $3$ 
of the multiplicative group $({\mathbb Z}/f{\mathbb Z})^*$.
Then,
\begin{equation}\label{SfabH}
s(a,b,f)
=\frac{f-1}{12f},
\ S(H,f)
=\frac{f-1}{12}
\hbox{ and }
N(f,H) 
=-1+12S(H,f)
=-1.
\end{equation}
\end{lemma} 

\begin{proof}
Noticing that 
$s(b,f,a) =s(b,b^2,a) =s(1,b,a) =s(b,a)$, 
by (\ref{bcd=abacd}),
and $s(f,a,b) =s(a^2,a,b) =s(a,1,b) =s(a,b)$, 
and using (\ref{scddc}), we obtain 
\begin{eqnarray*}
s(a,b,f)
&=&\frac{a^2+b^2+f^2-3\vert ab\vert f}{12abf}-s(b,f,a)-s(f,a,b)
\ \ \ \ \ \text{(by (\ref{sbcddcb}))}\\
&=&\frac{a^2+b^2+f^2-3\vert ab\vert f}{12abf}
-s(b,a)-s(a,b)\\
&=&\frac{a^2+b^2+f^2-3\vert ab\vert f}{12abf}
-\frac{a^2+b^2-3\vert ab\vert+1}{12ab}
=\frac{f-1}{12f}.
\end{eqnarray*}
Finally, $S(H,f)=s(1,f)+s(a,b,f)+s(b,a,f) =s(1,f)+2s(a,b,f)$ 
and use (\ref{s1d}) and (\ref{defNd0fHter}).
\end{proof}

\begin{Remark}\label{dedekindtwosizes}
Take $f_1=A^2+AB+B^2>0$, where $3\nmid f_1$ and $\gcd (A,B)=1$.
Set $f=(f_1+1)^3-1$. 
Then $f=a^2+ab+b^2$, 
where $a=Af_1+A-B$, $b=Bf_1+A+2B$ and $\gcd(a,b)=1$.
By Lemmas \ref{sabf} we have an infinite family of moduli $f$ for which the multiplicative group 
$({\mathbb Z}/f{\mathbb Z})^*$ contains at the same time 
an element $h=a/b$ of order $d=3$ for which $s(h,f)$ is asymptotic to $1/12$ 
and an element $h'=f_1+1$ of order $d=3$ for which $s(h',f)$ is asymptotic to $f^{2/3}/12$.
Indeed by \eqref{scddc} and \eqref{s1d} for $f=(f_1+1)^3-1$ we have
$s(h',f) =\frac{h'^5+h'^4-6h'^3+6}{12f}$. 
%Hence, as $f$ tends to infinity, $s(h',f)$ is asymptotic to $f^{2/3}/12$.
\end{Remark} 

To deal with the case $d_0>1$, we notice that by (\ref{defNd0fHter}) we have: 

\begin{proposition}\label{Nd0abf}
Let $d_0\geq 1$ be a given squarefree integer.
Take $f>3$ odd of the form $f =a^2+ab+b^2$, where $\gcd (a,b)=1$ and $\gcd (d_0,f)=1$. 
Set $H=\{1,a/b,b/a\}$, a subgroup of order $3$ of the multiplicative group $({\mathbb Z}/f{\mathbb Z})^*$. 
Let $N_{d_0}(f,H)$ be the rational number defined in (\ref{defNd0fH}).
Then 
$$N_{d_0}(f,H)
=N_{d_0}(f,\{1\})
+\frac{24\mu(d_0)}{\prod_{q\mid d_0}(q+1)}
\sum_{\delta\mid d_0}\frac{\delta\mu(\delta)}{\phi(\delta)}S(a,b,\delta f),$$
where $N_{d_0}(f,\{1\})$ is a rational number which depends only on $f$ modulo $d_0$, 
by Proposition \ref{H=1}, 
and where 
$$S(a,b,\delta f)
=\sum_{h\in ({\mathbb Z}/\delta f{\mathbb Z})^*\atop h\equiv a/b\pmod f} s(h,\delta f)
=\sum_{h\in ({\mathbb Z}/\delta f{\mathbb Z})^*\atop h\equiv b/a\pmod f} s(h,\delta f).$$
\end{proposition}

It seems that there are no explicit formulas for $S(a,b,\delta f)$, $S(H_\delta,\delta f)$ or $N_\delta(f,H)$ 
for $\delta>1$ 
(however, assuming that $b=1$ we will obtain such formulas in Section \ref{d3b1}
for $\delta\in\{2,3,6\}$). 
Instead, our aim is to prove in Proposition \ref{boundsSab} that $N_\delta(f,H) =O(\sqrt f)$ for $\delta\in\{2,3,6\}$.

Let $f>3$ be of the form 
$f=a^2+ab+b^2$, with $a,b\in {\mathbb Z}$ and $\gcd(a,b)=1$.
Hence, $a$ or $b$ is odd. 
Since 
$a^2+ab+b^2 
=a'^2+a'b'+b'^2$ 
$=a''^2+a''b''+b''^2$ 
and $a'/b' =a/b$ and $a''/b'' =a/b$ in $({\mathbb Z}/f{\mathbb Z})^*$, 
where $(a',b') =(-b,a+b)$ and $(a'',b'') =(-a-b,a)$, 
we may assume that both $a$ and $b$ are odd. 
Moreover, assume that $\gcd (3,f)=1$. 
If $3\nmid ab$, by swapping $a$ and $b$ as needed, 
which does not change neither $H$ nor $S(a,b,H)$, we may assume that 
$a\equiv -1\pmod 6$ and $b\equiv 1\pmod 6$.
If $3\mid ab$, by swapping $a$ and $b$ and then changing both $a$ and $b$ to their opposites as needed, 
which does not change neither $H$ nor $S(a,b,H)$, 
we may assume that $a\equiv 3\pmod 6$ and $b\equiv 1\pmod 6$. 
So in Proposition \ref{Nd0abf} we may restrict ourselves to the integers of the form 
\begin{multline}\label{deff}
f>3
\hbox{ is odd of the form } 
f=a^2+ab+b^2, 
\hbox{ with }
a,b\in {\mathbb Z}
\hbox{ odd }
\hbox{ and }
\gcd(a,b)=1\\
\hbox{and if $\gcd(3,f)=1$ then $a\equiv -1\hbox{ or }3\pmod 6$ and $b\equiv 1\pmod 6$}.
\end{multline}

\noindent\frame{\vbox{
\begin{proposition}\label{boundsSab}
Let $\delta\in\{2,3,6\}$ be given.
Let $f$ be as in (\ref{deff}), with $\gcd (f,\delta)=1$. 
Then, $s(h,\delta f) =O(\sqrt f)$ for any $h\in ({\mathbb Z}/\delta f{\mathbb Z})^*$ such that $h\equiv a/b\pmod f$. 
Consequently, for a given $d_0\in\{1,2,3,6\}$, in Proposition \ref{Nd0abf} we have $N_{d_0}(f,H) =O(\sqrt f)$, 
and we cannot expect great improvements on these bounds, by (\ref{N2fa1H}), (\ref{N3fa1H}) and (\ref{N6fa1H}).
\end{proposition}}}

\begin{proof}
First, by (\ref{SfabH}) we have
$$S(a,b,f) =s(a,b,f) =\frac{f-1}{12f}.$$
Second, $f$ being odd, recalling \eqref{Ad0f}
we have $A(2,f) =A(2,1) =0$, $N_2(f,\{1\}) =-1$, 
\begin{equation}\label{Sab2f}
S(a,b,2f) =s(a,b,2f)
\end{equation} 
and 
$$N_2(f,H) 
=-1
-8S(a,b,f)
+16S(a,b,2f).$$
Third, assume that $d_0\in\{3,6\}$. 
Then $\gcd(f,3)=1$. 
Hence, $f\equiv 1\pmod 6$. 
Therefore, 
$A(3,f) =A(3,1) =4/3$, $A(6,f) =A(6,1) =-4$,
$N_3(f,\{1\}) =N_6(f,\{1\}) =-1$,
$$N_3(f,H)
=-1-6S(a,b,f)
+9S(a,b,3f)$$
and 
$$N_6(f,H)
=-1+2S(a,b,f)-4S(a,b,2f)-3S(a,b,3f)+6S(a,b,6f).$$

\noindent If $a\equiv -1\pmod 6$, $b\equiv 1\pmod 6$ and $\delta\in\{1,2\}$, 
then 
$\{h\in ({\mathbb Z}/3\delta f{\mathbb Z})^*;\ h\equiv a/b\pmod f\} 
=\{a/b,(a+2f)/b\}$ 
and
\begin{equation}\label{Sab3deltaf1}
S(a,b,3\delta f)
=s(a,b,3\delta f)
+s(a+2f,b,3\delta f).
\end{equation}

\noindent If $a\equiv 3\pmod 6$, $b\equiv 1\pmod 6$ and $\delta\in\{1,2\}$, 
then 
$\{h\in ({\mathbb Z}/3\delta f{\mathbb Z})^*;\ h\equiv a/b\pmod f\} 
=\{(a-\delta f)/b,(a+\delta f)/b\}$ 
and
\begin{equation}\label{Sab3deltaf2}
S(a,b,3\delta f)
=s(a-\delta f,b,3\delta f)
+s(a+\delta f,b,3\delta f).
\end{equation}

\noindent Let us now bound the Dedekind-Rademacher sums in (\ref{Sab2f}), (\ref{Sab3deltaf1}) and (\ref{Sab3deltaf2}). 
We will need the bounds: 
\begin{equation}\label{boundsab}
\hbox{if } 
f=a^2+ab+b^2,
\hbox{ then }
\vert a\vert+\vert b\vert\leq\sqrt{4f}
\hbox{ and }
\vert ab\vert\geq\sqrt{f/3}.
\end{equation} 
Indeed, $4f-(\vert a\vert+\vert b\vert)^2 \geq 3(\vert a\vert-\vert b\vert)^2\geq 0$ 
and $f\leq a^2+a^2b^2+b^2 =3a^2b^2$. 

\noindent First, we deal with the Dedekind-Rademacher sums $s(a,b,\delta f)$ in (\ref{Sab2f}) and (\ref{Sab3deltaf1}), 
where $\delta\in\{2,3,6\}$. 
Here, $\gcd (a,b) =\gcd(a,\delta f)=\gcd (b,\delta f) =1$. 
Then (\ref{boundscd}) and (\ref{boundsab}) enable us to write (\ref{sbcddcb}) as follows: 
$$s(a,b,\delta f)+O(\sqrt f)+O(\sqrt f) =O(\sqrt f).$$ 
Hence, in (\ref{Sab2f}) and (\ref{Sab3deltaf1}) we have
$s(a,b,2f),\ s(a,b,3f),\ s(a,b,6f)=O(\sqrt f)$. 

\noindent Second, the remaining and more complicated Dedekind-Rademacher sums 
in (\ref{Sab3deltaf1}) and (\ref{Sab3deltaf2}) are of the form 
 $s(a+\varepsilon\delta f,b,3\delta f)$, 
where $\varepsilon\in\{\pm 1\}$, $\delta\in\{1,2\}$ and 
$\gcd (a+\varepsilon\delta f,3\delta f) 
=\gcd(b,3\delta f)=1$.
Set 
$\delta' =\gcd(a+\varepsilon\delta f,b)$. 
Then $\gcd(\delta',3\delta f)=1$. 
Thus, 
$s(a+\varepsilon\delta f,b,3\delta f)$ 
$=s((a+\varepsilon\delta f)/\delta',b/\delta',3\delta f)$, 
where now the three terms in this latter Dedekind-Rademacher are pairwise coprime. 
Then (\ref{boundscd}) and (\ref{boundsab}) enable us to write (\ref{sbcddcb}) as follows: 
\begin{multline*}
s((a+\varepsilon\delta f)/\delta',b/\delta',3\delta f) +O(\sqrt f) +s(b/\delta',3\delta f,(a+\varepsilon\delta f)/\delta')\\ 
=O(\delta'^2/b) =O(b) =O(\sqrt f).
\end{multline*}
Now, $3\delta f\equiv -3\varepsilon a\pmod {a+\varepsilon\delta f}$ gives 
$s(b/\delta',3\delta f,(a+\varepsilon\delta f)/\delta') =-\varepsilon s(b/\delta',3a,(a+\varepsilon\delta f)/\delta')$. 
Since the three rational integers in this latter Dedekind-Rademacher are pairwise coprime, 
the bounds (\ref{boundsab}) and (\ref{boundscd}) enable us to write (\ref{sbcddcb}) as follows: 
$$s(b/\delta',3a,(a+\varepsilon\delta f)/\delta')+O(\sqrt f)+O(\sqrt f) =O(\sqrt f).$$ 
It follows that 
$s(a+\varepsilon\delta f,b,3\delta f) 
=s((a+\varepsilon\delta f)/\delta',b/\delta',3\delta f)
=O(\sqrt f)$, 
i.e., in (\ref{Sab3deltaf1}) and (\ref{Sab3deltaf2}) we have
$s(a+2f,b,6f),\ s(a-2f,b,6f),\ s(a+2f,b,3f), s(a-f,b,3f), s(a+f,b,3f) =O(\sqrt f)$.
\end{proof}

\noindent\frame{\vbox{
\begin{Conjecture}\label{conjecturesab}
Let $\delta$ be a given square-free integer.
Let $f>3$ run over the odd integers of the form $f=a^2+ab+b^2$ 
with $\gcd(a,b)=1$ and $\gcd (\delta,f)=1$. 
Then $s(h,\delta f) =O(\sqrt f)$ for any $h\in ({\mathbb Z}/\delta f{\mathbb Z})^*$ such that $h\equiv a/b\pmod f$. 
Consequently, for a given square-free integer $d_0$, 
in Proposition \ref{Nd0abf}, we would have 
$N_{d_0}(f,H) =O(\sqrt f)$ for $\gcd (d_0,f)=1$.
\end{Conjecture}
}}

Putting everything together we obtain:

\noindent\frame{\vbox{
\begin{theorem}\label{thp3M2}
Let $p\equiv 1\pmod 6$ be a prime integer. 
Let $K$ be the imaginary subfield of degree $(p-1)/3$ of the cyclotomic field ${\mathbb Q}(\zeta_p)$.
Let $H$ be the subgroup of order $3$ of the multiplicative group $({\mathbb Z}/p{\mathbb Z})^*$. 
We have
\begin{equation*}
M(p,H)
={\pi^2\over 6}\left (1+\frac{N(p,H)}{p}\right )
={\pi^2\over 6}\left (1-\frac{1}{p}\right ) 
\hbox{ and } 
h_K^-
\leq 2\left (\frac{p}{24}\right )^{(p-1)/12},
\end{equation*}
and the following effective asymptotics and upper bounds
\begin{equation}\label{pc3M2}
M_2(p,H)
={\pi^2\over 8}\left (1+\frac{N_2(p,H)}{p}\right )
={\pi^2\over 8}\left (1+O(p^{-1/2})\right )
\hbox{ and } 
h_K^-
\leq 2\left (\frac{p+o(p)}{32}\right )^{\frac{(p-1)}{12}},
\end{equation} 
\begin{equation*}
M_6(p,H)
={\pi^2\over 9}\left (1+\frac{N_6(p,H)}{p}\right )
={\pi^2\over 9}\left (1+O(p^{-1/2})\right )
\hbox{ and } 
h_K^-
\leq 2\left (\frac{p+o(p)}{36}\right )^{\frac{(p-1)}{12}}.
\end{equation*} 
\end{theorem}
}}

\begin{proof}
The formulas for $M(p,H)$, $M_2(p,H)$ and $M_6(p,H)$ 
follow from (\ref{M(p,H)}), (\ref{Md0pH}), (\ref{SfabH}) and Proposition \ref{boundsSab}. 
The inequalities on $h_K^-$ are consequences as usual of (\ref{boundhKminusd0}) and Corollary \ref{PiKminus}.
\end{proof}

\subsection{The special case $p=a^2+a+1$ and $d_0=1,2,6$}\label{d3b1}
Let $f>3$ be of the form $f=a^2+a+1$, $a\in {\mathbb Z}$. 
Then $\gcd(f,6)=1$ if and only if $a\equiv 0,2,3,5\pmod 6$. 
We define $c_a'$, $c_a''$, $c_a'''$ and $c_a =(-1-2c_a'-c_a''+2c_a''')/12$, as follows:
$$\begin{array}{|c|r|r|r|r|r|}
\hline
a\pmod 6&c_a'&c_a''&c_a'''&c_a\\
\hline
0&-3a-2&-8a-5&-19a-10&-2a-1\\
1&3a+1&&&\\
2&-3a-2&8a+3&a-18&-3\\
3&3a+1&-8a-5&-a-19&-3\\
4&-3a-2&&&\\
5&3a+1&8a+3&19a+9&2a+1\\
\hline
\end{array}$$

\begin{theorem}\label{thp3M2M6}
Let $p\equiv 1\pmod 6$ be a prime integer of the form $p=a^2+a+1$ with $a\in {\mathbb Z}$. 
Let $K$ be the imaginary subfield of degree $(p-1)/3$ of the cyclotomic field ${\mathbb Q}(\zeta_p)$.
Let $H$ be the subgroup of order $3$ of the multiplicative group $({\mathbb Z}/p{\mathbb Z})^*$. 
We have 
\begin{equation}\label{pc3M2a}
M_2(p,H)
={\pi^2\over 8}\left (1-(-1)^{a}\frac{2a+1}{p}\right ),
\end{equation}
and 
\begin{equation}\label{pc3M6a}
M_6(p,H)
={\pi^2\over 9}\left (1+\frac{c_a}{p}\right ), 
\end{equation}
showing that the error term in (\ref{pc3M2}) is optimal.
\end{theorem}

\begin{proof}
The formula \eqref{pc3M2a} is a special case of \eqref{M2pd} for $d=3$.
By (\ref{Md0pH}), we have 
$$
M_6(p,H)
={\pi^2\over 9}\left (1+\frac{N_6(p,H)}{p}\right ).$$
Hence \eqref{pc3M6a} follows from Lemma \ref{sab6f} below.
\end{proof}

\begin{lemma}\label{sab2f}
Let $f>3$ be of the form $f=a^2+a+1$, $a\in {\mathbb Z}$. 
Set $H=\{1,a,a^2\}$, 
a subgroup of order $3$ of the multiplicative group $({\mathbb Z}/f{\mathbb Z})^*$. 
We have 
\begin{equation}\label{S2fa1H}
S(H,f)
=\frac{f-1}{12},
\ S(H_2,2f) 
=\frac{2f+c_a'}{12}
\end{equation} 
and 
\begin{equation}\label{N2fa1H}
N_2(f,H) 
=(-1)^{a-1}(2a+1).
\end{equation}
\end{lemma}

\begin{proof}
Apply Lemma \ref{(a^d-1)/(a-1)} with $d=3$ and $f=a^2+a+1$ to get \eqref{S2fa1H}.
Then using (\ref{defNd0fHbis}) 
we get
$N_2(f,H) 
=-f-4S(H,f)+8S(H_2,2f)
=\frac{2c_a'+1}{3}
=(-1)^{a-1}(2a+1)$. 
\end{proof} 

\begin{lemma}\label{sab6f}
Let $f>3$ be of the form $f=a^2+a+1$, $a\in {\mathbb Z}$. 
Assume that $\gcd(f,6)=1$, 
i.e. that $a\equiv 0, 2, 3, 5\pmod 6$.
Set $H=\{1,a,a^2\}$, 
a subgroup of order $3$ of the multiplicative group $({\mathbb Z}/f{\mathbb Z})^*$. 
Then 
\begin{equation}\label{S3fa1H}
S(H_3,3f) 
=\frac{5f+c_a''}{18},
\end{equation} 
\begin{equation}\label{N3fa1H}
N_3(f,H)
=\begin{cases}
-2a-1&\hbox{if $a\equiv 0\pmod 3$},\\
2a+1&\hbox{if $a\equiv 2 \pmod 3$},
\end{cases}
\end{equation}
\begin{equation}\label{S6fa1H}
S(H_6,6f) 
=\frac{10f+c_a'''}{18}
\end{equation}
and
\begin{equation}\label{N6fa1H}
N_6(f,H) 
=\begin{cases}
-2a-1&\hbox{if $a\equiv 0\pmod 6$,}\\
-3&\hbox{if $a\equiv 2,3\pmod 6$,}\\
2a+1&\hbox{if $a\equiv 5\pmod 6$.}
\end{cases}
\end{equation}
\end{lemma}

\begin{proof}
Let us for example detail the computation of $S(H_6,6f)$ in the case that $a\equiv 0\pmod 6$. 
We have 
$f\equiv 1\pmod 6$
and $H_6=\{1,1+4f,a+f,a+5f,a^2+f,a^2+5f\}$. 
Since $a^2+f =(a+f)^{-1}$ and $a^2+5f=(a+5f)^{-1}$ in $({\mathbb Z}/f{\mathbb Z})^*$, 
we have $S(H_6,6f) =s(1,6f)+s(1+4f,6f)+2s(a+f,6f)+2s(a+5f,6f)$, 
by (\ref{cc*}). 
Using (\ref{scddc}) and (\ref{s1d}) we obtain 
$s(1,6f) =\frac{18f^2-9f+1}{36f}$,
$s(1+4f,6f) =\frac{2f^2-13f+1}{36f}$, 
$s(a+f,6f) =-\frac{(3a-21)f+1}{72f}$ 
and $s(a+5f,6f) =-\frac{(35a+19)f+1}{72f}$. 
Formula (\ref{S6fa1H}) follows. 

By (\ref{defNd0fHbis}), we have 
$$N_3(f,H)
=-f-3S(H,f)+\frac{9}{2}S(H_3,3f)$$ 
and 
$$N_6(f,H)
=-f
+S(H,f)-2S(H_2,2f)-\frac{3}{2}S(H_3,3f)+3S(H_6,6f).$$
Using (\ref{SfabH}), (\ref{S2fa1H}) and (\ref{S3fa1H}), 
we obtain 
$N_3(f,H)
=\frac{c_a''+1}{4}$ 
and \eqref{N3fa1H}. 
Using (\ref{SfabH}), (\ref{S2fa1H}), (\ref{S3fa1H}) and (\ref{S6fa1H}), 
we obtain 
$N_6(f,H)
=\frac{-1-2c_a'-c_a''+2c_a'''}{12}
=c_a$ 
and \eqref{N6fa1H}.
\end{proof}

\section{Conclusion and a conjecture}\label{conclusion}

The proof of Lemma \ref{(a^d-1)/(a-1)} gives for $d\geq 3$ odd and $a\neq 0,\pm 1$
\begin{equation}\label{s(a,f)}
s\left (a,\frac{a^d-1}{a-1}\right )
=\frac{(f-1)(f-a^2-1)}{12af}
=O\left (f^{1-\frac{1}{d-1}}\right ).
%\hbox{ for $d\geq 3$ odd and $a\neq 0,\pm 1$.}
\end{equation}
Our numerical computations suggest the following stronger version of Theorem \ref{indivbound}: 

\begin{Conjecture}\label{conjDedekind}
There exists $C>0$ such that for 
any odd $d>1$ dividing $p-1$ 
and any $h$ of order $d$ in the multiplicative group $({\mathbb Z}/p{\mathbb Z})^*$ we have
\begin{equation}\label{conjecture}
\vert s(h,p)\vert
\leq Cp^{1-\frac{1}{\phi(d)}}.
\end{equation} 
\end{Conjecture}

Indeed, for $p\leq 10^6$ we checked on a desk computer that 
any odd $d>1$ dividing $p-1$ 
and any $h$ of order $d$ in the multiplicative group $({\mathbb Z}/p{\mathbb Z})^*$ we have
$$Q(h,p)
:=\frac{\vert s(h,p)\vert}{p^{1-\frac{1}{\phi(d)}}}
\leq Q(2,2^7-1)
=0.08903\cdots$$ 
The estimate \eqref{conjecture} would allow to slightly extend the range of validity of Theorem \ref{asympd0} 
to $d \leq (1-\varepsilon)\frac{\log p}{\log \log p}$. 
Moreover the choice $a=2$ in (\ref{s(a,f)}) for which $s(2,f)$ is asymptotic to $\frac{1}{24}f$ with $f=2^d-1$ 
shows that $s(h,p) =o(p)$ cannot hold true in the range $d \asymp \log p$.
Notice that we cannot expect a better bound than (\ref{conjecture}), by (\ref{s(a,f)}).
Finally, the restriction that $p$ be prime in (\ref{conjecture}) is paramount by Remark \ref{dedekindtwosizes} 
where $s(a,f) \sim f^{2/3}/12$ for $a$ of order $3$ in $({\mathbb Z}/(a^3-1){\mathbb Z})^*$.

\section*{Acknowledgements}
This work was supported by the Ministero della Istruzione e della Ricerca Young Researchers Program Rita Levi Montalcini (M.M.).

{\small
\bibliography{central}

}

\end{document}